\documentclass[extended]{svjour3}

\usepackage{graphicx}
\usepackage{subfigure}
\usepackage{amsmath}
\usepackage{amssymb}
\usepackage{theorem}
\theorembodyfont{\upshape}
\usepackage{algorithm}
\usepackage{algorithmic}

 \usepackage{mathptmx}      
\journalname{}
\begin{document}

\title{A convergence framework for inexact nonconvex and nonsmooth  algorithms   and its applications to several  iterations}



\author{Tao Sun  \and Hao Jiang \and Lizhi Cheng\and Wei Zhu
}


\institute{Tao Sun \at
              Department of Mathematics, National University of Defense Technology, Changsha, Hunan,  410073, P.R.China.\\
              \email{nudtsuntao@163.com}           
            \and
           Hao Jiang \at
              College of Computer, National University of Defense  Technology, Changsha, Hunan,  410073, P.R.China.
                \and   Lizhi Cheng \at
              The State Key Laboratory for High Performance Computation, National University of Defense  Technology, Changsha, Hunan,  410073, P.R.China.
              \and Wei Zhu \at Hunan Key Laboratory for Computation and Simulation in Science and Engineering, School of Mathematics and Computational Science, Xiangtan University, Xiangtan, Hunan, 411105, P.R.China.
}

\date{Received: date / Accepted: date}

\maketitle
\begin{abstract}
In this paper, we consider the convergence of an abstract  inexact nonconvex and nonsmooth algorithm, which cannot be included in previous framework. We promise  a \textit{pseudo sufficient descent condition} and a \textit{pseudo relative error condition}, which  are both related to an auxiliary sequence, for the algorithm;  and a continuity condition is assumed to hold. In fact, a lot of classical  inexact nonconvex and nonsmooth algorithms allow these three conditions. Under an assumption on the auxiliary sequence, we prove the sequence generated by the general algorithm converges to a critical point of the objective function if being assumed semi-algebraic property. The core of the proofs lies in building a new Lyapunov function, whose successive difference provides a bound for the successive   difference of the  points generated by the algorithm.  And then, we apply our findings to several classical inexact nonconvex iterative algorithms and derive the corresponding convergence results.
\end{abstract}



\tableofcontents
\section{Introduction}
Minimization of the nonconvex and nonsmooth function
\begin{equation}\label{model}
    \min_x F(x)
\end{equation}
 is a core part of nonlinear programming and applied mathematics.  Different from   searching the global minimizers in the convex community, the global convergence is impossible for the general nonconvex   problems. Instead, the nonconvex optimization focuses on the critical point convergence, that is, proving $\textrm{dist}(\textbf{0},\partial F(x^{k}))\rightarrow 0$ and where  $\partial F$ denotes the limiting  subdifferential of $F$ (see definition in Sec. 2).
 A natural problem is then, does the  generated sequence $(x^k)_{k\geq 0}$ converge? In general cases, the answer is still unclear.
Fortunately, in most practical cases, the objective functions
enjoy the semi-algebraic property (see definition in Sec. 2), which can help to make the  sequence convergence be possible.

In paper \cite{attouch2013convergence},  with  semi-algebraic property  assumption on $F$, the authors present a sequence convergence framework: for the sequence $(x^k)_{k\geq 0}$ generated by a very general scheme for problem (\ref{model}),    three conditions,  \textit{sufficient descent condition},  \textit{relative error condition} and \textit{continuity condition},  are assumed to hold; mathematically, these conditions can be presented as: for some $a>0,c>0$
\begin{eqnarray}\label{oldconditions}
\left\{\begin{array}{c}
             F(x^{k})-F(x^{k+1})\geq a\|x^{k+1}-x^k\|^2,\\
    \textrm{dist}(\textbf{0},\partial F(x^{k+1}))\leq c\|x^{k+1}-x^k\|,\\
 \textrm{for}~\textrm{any}~\textrm{stationary~point}~x^*~\textrm{and}~\textrm{subsequence}~(x^{k_j})_{j\geq 0}\rightarrow x^*,~\textrm{~it~holds}~F(x^{k_j})\rightarrow F(x^*).
       \end{array}
\right.
\end{eqnarray}
It then can be proved that $(x^k)_{k\geq 0}$ is convergent to some $x^*$ satisfying that $\textrm{dist}(\textbf{0},\partial F(x^{*}))=0$.
With this framework, for given algorithms, one just needs to check whether  the objective function is semi-algebraic and condition \eqref{oldconditions} holds  or not. Better yet, semi-algebraicity is common and  various classical algorithms  admit  the proposed three conditions.   Therefore, such a framework can be widely used. However, condition \eqref{oldconditions}  is ``non-robust"; specifically, for one algorithm satisfying \eqref{oldconditions}, its slight variation, like inexact version, is usually a violation. In fact, many   inexact methods   break   condition \eqref{oldconditions}; and then, it is necessary to build novel framework for these inexact algorithms and corresponding convergence results.

 \subsection{Motivation: inexact algorithms excluded in previous framework}
 We present a toy example of an inexact algorithm that fails to satisfy the   sufficient descent condition and relative error condition:
 let $F$ be differentiable but nonconvex, the   gradient descent for minimizing $F$ is then
\begin{equation}\label{gradexam}
    x^{k+1}=x^k-h\cdot\nabla F(x^k).
\end{equation}
If the gradient of $F$ is Lipschitz with $L$ and $0<h<\frac{1}{L}$, the sequence $(x^k)_{k\geq 0}$ generated by (\ref{gradexam}) satisfies condition (\ref{oldconditions}). However, if the iteration is corrupted by some noise $e^k$ in each step, i.e.,
\begin{equation}\label{igradexam}
    x^{k+1}=x^k-h\cdot\nabla F(x^k)+e^k.
\end{equation}
 It can be easily checked that the sequence $(x^k)_{k\geq 0}$ generated by (\ref{igradexam}) is likely violating  the first and second conditions in  (\ref{oldconditions}) when $e^k\neq \textbf{0}$. The existing analysis framework then cannot be   used for the algorithm (\ref{igradexam}).
To use previous routine, the authors in \cite{attouch2013convergence} proposed an assumption for the noise: the noise should be bounded by the successive difference of the iteration,  that is,
\begin{equation}\label{jia}
    \|e^k\|\leq \ell\cdot\|x^{k+1}-x^k\|,
\end{equation}
where $\ell>0$.
The conditions (\ref{oldconditions}), with assumption  (\ref{jia}), then can all hold; and the sequence convergence is obviously  trivial.    Such an assumption, however,  indicates that $e^k$, which can be  governed in each iteration, is rather than ``real" noise. A more realistic assumption, which is also frequently used in the convex cases, is about the summability of the  noise.
Motivated by this, in this paper, we get rid of the dependent assumption like (\ref{jia}) and turn to use another summable assumption.
Although the convergence can be easily derived with summable assumption for the nonconvex algorithms,  the existing results   all investigate whether  $\left(\textrm{dist}(\textbf{0},\partial F(x^{k}))\right)_{k\geq 0}$ goes to zero. In this paper, we study the sequence convergence, i.e., whether $(x^k)_{k\geq 0}$ converges or not.

\subsection{Novel convergence framework, assumption and proof}
Although the inexact algorithms always fail to obey the first two of the core condition (\ref{oldconditions}) when losing assumption  (\ref{jia}), we find that many of them satisfy   alternative conditions:
\begin{eqnarray}\label{condition}
\left\{\begin{array}{c}
         F(x^{k})-F(x^{k+1})\geq a\|\omega^{k+1}-\omega^k\|^2-b\eta_k^2, \\
         \textrm{dist}(\textbf{0},\partial F(x^{k+1}))\leq c\sum_{j=k-\tau}^k\|\omega^{j+1}-\omega^j\|+d\eta_k, \\
 \textrm{for}~\textrm{any}~\textrm{stationary~point}~\hat{x}~\textrm{and}~\textrm{subsequence}~(x^{k_j})_{j\geq 0}\rightarrow \hat{x},~\textrm{~it~holds}~F(x^{k_j})\rightarrow F(\hat{x}),
       \end{array}
\right.
\end{eqnarray}
where $a,b,c,d>0$ are constants, and $(\eta_k)_{k\geq 0}$ is a nonnegative sequence, and $\tau\in \mathbb{Z}^{+}_0$ and $(\omega^k)_{k\geq 0}$ is a sequence satisfying
\begin{equation}\label{pointass}
    \sum_{k}\|\omega^k-\omega^{k+1}\|<+\infty\Rightarrow \sum_{k}\|x^k-x^{k+1}\|<+\infty.
\end{equation}
The continuity condition is kept here. Obviously, if $\eta_k\equiv 0$, $\omega^k\equiv x^k$ and $\tau=0$, the condition will reduce to (\ref{oldconditions}).  In the new conditions, two extra    variables $(\eta_k)_{k\geq 0}$ and $(\omega^k)_{k\geq 0}$ are introduced; the first one actually represents the length of the noise multiplied by a constant in the $k$th iteration,   and the other one is a   composition of $(x^k)_{k\geq 0}$.

A novel   assumption on $(\eta_k)_{k\geq 0}$, which is actually a kind of summable requirement,  is recruited for the sequence convergence.  Noticing that $\eta_k$ is closely related to the norm of the noise, thus, such an assumption coincides with the traditional background of inexact algorithms. Our proof is based on the novel framework and assumption, and of course, the  semi-algebraic property.

The core of the proof lies in using an auxiliary function whose successive difference  gives a bound to the successive difference of the sequence $\|\omega^{k+1}-\omega^k\|^2$. If $F$ is semi-algebraic, the auxiliary function is then  Kurdyka-$\L$ojasiewicz (short as K{\L}, see definition in Sec. 2). And then, we build  sufficient descent involving  the auxiliary function  and $\|\omega^{k+1}-\omega^k\|^2$. We denote $t_k$, which is a composition of $(\eta_k)_{k\geq 0}$, in (\ref{denote-t}). In the $(k+1)$-th iteration, the distance between  subdifferential of the new function and the origin is bounded by the composition of $\|\omega^{k+1}-\omega^k\|$, $t_k$ and $t_{k+1}$. And then, we prove the finite length of $(x^k)_{k\geq 0}$ provided $(t_k)_{k\geq 0}$ is also summable. Another technique is the use of the parameter to be determined in the auxiliary function. Based on this trick, we can simplify the noise requirement and  prove that $\eta_k=\mathcal{O}(\frac{1}{k^{\alpha}})$ with $\alpha>1$ is sufficient for the sequence convergence.

\subsection{Related work}
Recently, the convergence analysis in nonconvex optimization has paid increasing attention to using the semi-algebraic property in proofs. In paper \cite{attouch2009convergence}, the authors proved the convergence of proximal algorithm minimizing the semi-algebraic functions. In \cite{attouch2009convergence}, the  rates for the iteration converging to a critical point were exploited. An alternating proximal algorithm was considered in \cite{attouch2010proximal}, and the convergence was proved under semi-algebraic assumption on the objective function. Latter, a   proximal linearized alternating minimization algorithm was proposed and studied in \cite{bolte2014proximal}. A convergence framework was given in \cite{attouch2013convergence}, which contains various nonconvex algorithms. In \cite{frankel2015splitting}, the authors modified the framework for analyzing splitting methods with variable metric, and  proved the general convergence rates. The nonconvex ADMM  was studied under semi-algebraic assumption by \cite{li2015global,li2016douglas}. And latter paper \cite{sun2017pre} proposed the nonconvex primal-dual algorithm and proved the convergence. The semi-algebraic  analysis convergence method was applied to analyzing the convergence of the reweighted algorithm by \cite{sun2017global}. And the extension to the reweighted nuclear norm version was developed in \cite{sun2017convergence}. Recently, the DC algorithm has also employed the semi-algebraic property in the convergence analysis \cite{an2017convergence}.
\subsection{Contribution and organization}
In this paper, we focus on the sequence convergence of inexact nonconvex algorithms, which are dealed with  under a strong assumption on the noise in previous work. A much milder assumption, which is frequently used in inexact algorithm analysis, is introduced. We first propose a novel framework (\ref{condition}), which is more general than the frameworks proposed in \cite{attouch2013convergence} and \cite{frankel2015splitting}, and more importantly, contains plenty of algorithms that excluded in previous frameworks.    The convergence is proved for any sequence satisfying (\ref{condition}) and \eqref{pointass} with $\eta_k=\mathcal{O}(\frac{1}{k^{\alpha}})$ and $\alpha>1$ if $F$ is a semi-algebraic function. A new analysis method is developed, in which we employ an auxiliary Lyapunov function, which is a composition of the $F$ and the length of the noise.  With our results, for a specific algorithm, we just need to verify whether conditions (\ref{condition}) and \eqref{pointass} hold or not. In the application part,   the corresponding sequence convergence guarantees  are provided for several classical inexact nonconvex algorithms.

The rest of the paper is organized as follows. In section 2, we list necessary  preliminaries. Section 3 contains the main results. In section 4, we provide the applications. Section 5 concludes the paper.

\section{Preliminaries}
This section contains two parts: in the first subsection, we introduce the basic definitions and properties of subdifferentials; in the second subsection, the semi-algebraic and K{\L} properties are introduced.
\subsection{Subdifferential}
 Given an lower semicontinuous function $J: \mathbb{R}^N\rightarrow (-\infty,+\infty]$, its domain is defined by
$$\textrm{dom} (J):=\{x\in \mathbb{R}^N: J(x)<+\infty\}.$$
The notion of subdifferential plays a central role in variational analysis.

\begin{definition}[subdifferential] Let  $J: \mathbb{R}^N \rightarrow (-\infty, +\infty]$ be a proper and lower semicontinuous function.
\begin{enumerate}
  \item For a given $x\in \textrm{dom} (J)$, the Fr$\acute{e}$chet subdifferential of $J$ at $x$, written $\hat{\partial}J (x)$, is the set of all vectors $u\in \mathbb{R}^N$   satisfying
  $$\lim_{y\neq x}\inf_{y\rightarrow x}\frac{J(y)-J(x)-\langle u, y-x\rangle}{\|y-x\|}\geq 0.$$
When $x\notin \textrm{dom}(J)$, we set $\hat{\partial}J(x)=\emptyset$.

\item The (limiting) subdifferential, or simply the subdifferential, of $J$ at $x\in \textrm{dom}(J)$, written $\partial J(x)$, is defined through the following closure process
$$\partial J(x):=\{u\in\mathbb{R}^N: \exists x^k\rightarrow x, J(x^k)\rightarrow J(x)~\textrm{and}~u^k\in \hat{\partial}J(x^k)\rightarrow u~\textrm{as}~k\rightarrow \infty\}.$$
\end{enumerate}
\end{definition}
It is easy to verify that the Fr$\acute{e}$chet subdifferential is convex and closed while the subdifferential is closed. When $J$ is convex,  the definition agrees with the subgradient in convex analysis  as
$$\partial J(x):=\{v: J(y)\geq J(x)+\langle v,y-x\rangle~~\textrm{for}~~\textrm{any}~~y\in \mathbb{R}^N\}.$$
The graph of  subdifferential for a real extended valued function $J: \mathbb{R}^N \rightarrow (-\infty, +\infty]$ is defined by
$$\textrm{graph} (\partial J):=\{(x,v)\in\mathbb{R}^N\times\mathbb{R}^N: v\in \partial J(x)\}.$$
And the domain of the subdifferential of $\partial J$ is given as
$$\textrm{dom}(\partial J):=\{x\in \mathbb{R}^N : \partial J(x)\neq \emptyset\}.$$
Let $\{(x^k,v^k)\}_{k\in \mathbb{N}}$ be a sequence in $\mathbb{R}^N\times \mathbb{R}^N$ such that $(x^k,v^k)\in \textrm{graph }(\partial J)$. If $(x^k,v^k)$  converges to $(x, v)$ as $k\rightarrow +\infty$ and $J(x^k)$ converges to $v$ as $k\rightarrow +\infty$, then $(x, v)\in \textrm{graph }(\partial J)$.
A necessary condition for $x\in\mathbb{R}^N$ to be a minimizer of $J(x)$ is
\begin{equation}\label{Fermat}
\textbf{0}\in \partial J(x).
\end{equation}
When $J$ is convex, (\ref{Fermat}) is also sufficient. A point that satisfies (\ref{Fermat}) is called (limiting) critical point. The set of critical points of $J(x)$ is denoted by $\textrm{crit}(J)$.  More details about the definition of subdifferential can be found in the textbooks \cite{rockafellar2009variational,rockafellar2015convex}.
\subsection{Semi-algebraic property and Kurdyka-{\L}ojasiewicz function}
With the definition of subdifferential, we now are prepared to introduce the  Kurdyka-{\L}ojasiewicz (K{\L}) property and function.
\begin{definition}\label{KL}\cite{lojasiewicz1993geometrie,kurdyka1998gradients,bolte2007lojasiewicz}
(a) The function $J: \mathbb{R}^N \rightarrow (-\infty, +\infty]$ is said to have the  Kurdyka-{\L}ojasiewicz (K{\L}) property at $\overline{x}\in \textrm{dom}(\partial J)$ if there
 exist $\eta\in (0, +\infty]$, a neighborhood $U$ of $\overline{x}$ and a continuous concave function $\varphi: [0, \eta)\rightarrow \mathbb{R}^+$ such that
\begin{enumerate}
  \item $\varphi(0)=0$.
  \item $\varphi$ is $C^1$ on $(0, \eta)$.
  \item for all $s\in(0, \eta)$, $\varphi^{'}(s)>0$.
  \item for all $x$ in $U\bigcap\{x|J(\overline{x})<J(x)<J(\overline{x})+\eta\}$, the Kurdyka-{\L}ojasiewicz inequality holds
\begin{equation}
  \varphi^{'}(J(x)-J(\overline{x}))\cdot \textrm{dist}(\textbf{0},\partial J(x))\geq 1.
\end{equation}
\end{enumerate}

(b) Proper lower semicontinuous functions which satisfy the K{\L} inequality at each point of $\textrm{dom}(\partial J)$ are called K{\L}  functions.
\end{definition}

It is hard to directly judge whether a function is  K{\L} or not. Fortunately, the concept of semi-algebraicity can help to find and check a very rich class of K{\L} functions.

\begin{definition}[Semi-algebraic sets and functions \cite{bolte2007lojasiewicz,bolte2007clarke}]

(a) A subset $S$ of $\mathbb{R}^N$ is a real semi-algebraic set if there exists a finite number of real polynomial functions $g_{ij}, h_{ij}:\mathbb{R}^N\rightarrow \mathbb{R}$ such that
$$S=\bigcup_{j=1}^p\bigcap_{i=1}^q\{u\in \mathbb{R}^N:g_{ij}(u)=0~\textrm{and}~~h_{ij}(u)<0\}.$$

(b) A function $h:\mathbb{R}^N\rightarrow (-\infty, +\infty]$ is called semi-algebraic if its graph
$$\{(u, t)\in \mathbb{R}^{N+1}: h(u)=t\}$$
is a semi-algebraic subset of $\mathbb{R}^{N+1}$.
\end{definition}
Better yet, the semi-algebraicity enjoys many quite nice properties \cite{bolte2007lojasiewicz,bolte2007clarke}. Various common functions and sets are also semi-algebraic; we just put a few of them here:
\begin{itemize}
\item Real polynomial functions.
\item Indicator functions of semi-algebraic sets.
\item Finite sums and product of semi-algebraic functions.
\item Composition of semi-algebraic functions.
\item Sup/Inf type function, e.g., $\sup\{g (u, v) : v \in C\}$ is semi-algebraic when $g$ is a semi-algebraic
function and $C$ a semi-algebraic set.
\item Cone of PSD matrices, Stiefel manifolds and constant rank matrices.
\end{itemize}

\begin{lemma}[\cite{bolte2007lojasiewicz}]\label{semikl}
Let $J:\mathbb{R}^N\rightarrow \mathbb{R}$ be a proper and closed function. If $J$ is semi-algebraic then it satisfies the K{\L} property
at any point of $\textrm{dom} (\partial J)$.
\end{lemma}

The proofs in \cite{attouch2013convergence} use a \textit{local area analysis}; the authors first prove that the sequence falls into a neighbor of some point  after enough iterations and then employ the K$\L$ property around the point. In latter paper \cite{bolte2014proximal}, the authors prove a uniformed K$\L$ lemma for a closed set and much simplify the proofs.
Now we present a lemma for the uniformized K{\L} property. With this lemma, we can make the proofs much more concise.

\begin{lemma}[\cite{bolte2014proximal}]\label{con}
Let $J: \mathbb{R}^N\rightarrow \mathbb{R}$ be a proper lower semi-continuous function and $\Omega$ be a compact set. If $J$ is a constant on $\Omega$ and $J$ satisfies the K{\L} property at each point on $\Omega$, then there exists concave function $\varphi$ satisfying the four assumptions in Definition \ref{KL}  and $\delta,\varepsilon>0$ such that for any $\overline{x}\in \Omega$ and any $x$ satisfying that $\emph{dist}(x,\Omega)<\varepsilon$ and $J(\overline{x})<J(x)<J(\overline{x})+\delta$, it holds that
\begin{equation}
    \varphi^{'}(J(x)-J(\overline{x}))\cdot\emph{dist}(\textbf{0},\partial J(x))\geq 1.
\end{equation}
\end{lemma}
\section{Convergence analysis}
The sequence $(\eta_k)_{k\geq 0}$ is assumed to satisfy
\begin{equation}\label{noise}
    \sum_{k}\eta_k<+\infty.
\end{equation}
It is worth mentioning that the assumption \eqref{noise} is a necessary condition for the  guarantee of the \textit{sequence convergence} in general case. To see this, we consider the inexact gradient descent example \eqref{igradexam} in a very special case that $F\equiv 0$. And then, we get $x^k=x^0+\sum_{i=0}^{k-1} e^i$. Further, we consider the one-dimensional case, in which, we set    $e^k=\eta_k$. And then, it holds that $x^k=x^0+\sum_{i=0}^{k-1} \eta_i$. It is easy to see, in this example, $(x^k)_{k\geq 0}$ will diverge if \eqref{noise} fails to hold. However, in our proofs, only \eqref{noise} barely promises the sequence convergence.  Actually, the final assumption  is  a little stronger than \eqref{noise}.

Now, we introduce the Lyapunov function used in the analysis. Given any fixed $\theta>1$ to be determined, we denote a new function as
\begin{equation}
    \xi(z):=F(x)+\frac{t^{\theta}}{\theta},\,~ z:=(x,t)\in \mathbb{R}^{N+1}.
\end{equation}
We also need to define the new sequences as
\begin{equation}\label{denote-t}
    t_k:=\big(\theta\cdot b\cdot \sum_{l=k}^{+\infty}\eta_l^2\big)^{\frac{1}{\theta}},\,~ z^{k}:=(x^k,t_k).
\end{equation}
Due to that $ \sum_{l=k'}^{+\infty}\eta_l^{2}\leq \sum_{l=k'}^{+\infty}\eta_l<+\infty$ when   $k'$ is larger enough, $t_k$ is well-defined.
The aim in this part is  proving that $(z^k)_{k\geq 0}$ generated by the  algorithm converges to a critical point of $\xi$, and building the relationships between the critical points of $\xi$ and $F$.  The proof contains two main steps:
\begin{enumerate}
  \item  Find a positive constant $\rho_1$ such that
  $$ \rho_1 \|\omega^{k+1}-\omega^k\|^2\leq \xi(z^k)-\xi(z^{k+1}),  \,~k=0,1, \cdots.$$
  \item Find another positive constants $\rho_2,\rho_3,\rho_4$  such that
        $$\textrm{dist}(\textbf{0},\partial \xi(z^{k+1}))\leq \rho_2\sum_{j=k-\tau}^k\|\omega^{j+1}-\omega^j\|+\rho_3 \eta_k+\rho_4(t_{k+1})^{\theta-1}, ~\, k=0,1, \cdots.$$
\end{enumerate}

\begin{lemma}\label{descend}
Assume that $\{x^{k}\}_{k=0,1,2,\ldots}$ is generated by the  general inexact  algorithm satisfying conditions (\ref{condition}) and (\ref{pointass}), and condition (\ref{noise}) holds.
Then, we have the following results.\\

(1)  It holds that
\begin{equation}\label{descend1}
    \xi(z^k)-\xi(z^{k+1})\geq a\|\omega^{k}-\omega^{k+1}\|^2.
\end{equation}
And then, $(z^k)_{k\geq 0}$ is bounded if $F$ is coercive.

(2) $\sum_{k}\|x^{k+1}-x^{k}\|^2<+\infty$, which implies  that
\begin{equation}
    \lim_{k}\|x^{k+1}-x^k\|=0.
\end{equation}

\end{lemma}
\begin{proof}
(1) From the direct algebra computations, we can easily obtain
\begin{eqnarray}\label{lema1temp1}
\xi(z^k)-\xi(z^{k+1})&=&F(x^k)-F(x^{k+1})+\frac{t_k^{\theta}-t^{\theta}_{k+1}}{\theta}\nonumber\\
&=&F(x^k)-F(x^{k+1})+b\eta_k^2\nonumber\\
&\geq& a\|\omega^{k}-\omega^{k+1}\|^2.
\end{eqnarray}
If $F$ is coercive, then $\xi$ is also coercive. Thus, $(z^k)_{k\geq 0}$ is bounded  due to the boundedness of  $(\xi(z^k))_{k\geq 0}$.

(2) From (\ref{descend1}), $\{\xi(z^{k})\}_{k=0,1,2,\ldots}$ is descending. Note that $\inf\,\xi>-\infty$, $\{\xi(z^{k})\}_{k=0,1,2,\ldots}$ is convergent. Hence, we can easily have
$$\sum_{n=0}^{k}\|\omega^{n+1}-\omega^{n}\|^2\leq \frac{\xi(z^{0})-\xi(z^{k+1})}{a}<+\infty.$$
With (\ref{pointass}), we then prove the result.
\end{proof}

\begin{lemma}\label{lemma3}
If the conditions of Lemma \ref{descend} hold, then
\begin{equation}
\emph{dist}(\textbf{0},\partial \xi(z^{k+1}))\leq c\sum_{j=k-\tau}^k\|\omega^{j+1}-\omega^j\|+d\eta_k+t_{k+1}.
\end{equation}
\end{lemma}
\begin{proof}
Direct calculation yields
\begin{equation}
    \partial \xi(z^{k+1})=\left(
                            \begin{array}{c}
                              \partial F(x^{k+1}) \\
                              (t_{k+1})^{\theta-1} \\
                            \end{array}
                          \right)
    .
\end{equation}
 Thus, we have
 \begin{eqnarray}
    \textrm{dist}(\textbf{0},\partial \xi(z^{k+1}))&\leq&\textrm{dist}(\textbf{0},\partial F(x^{k+1}))+(t_{k+1})^{\theta-1}\nonumber\\
    &\leq&c\sum_{j=k-\tau}^k\|\omega^{j+1}-\omega^j\|+d\eta_k+(t_{k+1})^{\theta-1}.
 \end{eqnarray}
\end{proof}

In the following, we establish some results about the limit points of the sequence generated by the general algorithm.
We need a  definition about the  limit point which is introduced in \cite{attouch2013convergence}.
\begin{definition}
For a sequence $\{d^k\}_{k=0,1,2,\ldots}$,
define that
$$\mathcal{M}(d^0):=\{d\in \mathbb{R}^{N}: \exists~\textrm{an increasing sequence of integers}~\{k_j\}_{j\in\mathbb{N}}~
\textrm{such that} ~d^{k_j}\rightarrow d ~\textrm{as}~ j\rightarrow \infty\},$$
 where $d^0\in \mathbb{R}^{N}$ is the starting point.
\end{definition}

\begin{lemma}\label{points}
Suppose that  $\{z^{k}=(x^k,t_k)\}_{k=0,1,2,\ldots}$ is generated by general algorithm and $F$ is coercive.  And the conditions of Lemma \ref{descend} hold. Then, we have the following results.

(1) For any $z^*=(x^*,t^*)\in \mathcal{M}(z^0)$, we have $t^*=0$ and $\xi(z^*)=F(x^*)$.

(2) $\mathcal{M}(z^0)$ is nonempty and $\mathcal{M}(z^0)\subseteq \emph{crit}(\xi)$.

(2') $\mathcal{M}(x^0)$ is nonempty and $\mathcal{M}(x^0)\subseteq \emph{crit}(F)$

(3) $\lim_{k}\emph{dist}(z^k,\mathcal{M}(z^0))=0$.

(3') $\lim_{k}\emph{dist}(x^k,\mathcal{M}(x^0))=0$.

(4) The function $\xi$ is finite and constant on $\mathcal{M}(z^0)$.

(4') The function $F$ is finite and constant on $\mathcal{M}(x^0)$.

\end{lemma}
\begin{proof}
(1) Noting $(t_k)_{k\geq 0}\rightarrow 0$, $t^*=0$ and
$$\xi(z^*)=\xi(x^*,0)=F(x^*).$$

(2) It is easy to see the coercivity of $\xi$. With Lemma \ref{descend} and the coercivity of $\xi$, $(z^{k})_{k\geq 0}$  is bounded. Thus, $\mathcal{M}(z^0)$ is nonempty. Assume that $z^*\in \mathcal{M}(z^0)$, from the definition, there exists
a subsequence $z^{k_i}\rightarrow z^*$.  From Lemmas \ref{descend} and \ref{lemma3}, we have
$\textrm{dist}(\textbf{0},\partial\xi(z^{k_i}))\rightarrow \mathbf{0}$. The closedness of $\partial \xi$ indicates that $\mathbf{0}\in \partial \xi(z^*)$, i.e. $z^*\in \textrm{crit}(\xi)$.

(2') With the facts $z=(x,t)$ and $\xi(z)=F(x)+\frac{t^{\theta}}{\theta}$, we can easily derive the results.

(3)(3') This item follows as a consequence of the definition of the limit point.

(4) Let $l$ be the limit of $(\xi(x^k))_{k\geq 0}$. Let $z^*=(x^*,t^*)$  be any one   point in $\mathcal{M}(z^0)$. With (1), we have $t^*=0$. Let $(z^{k_j})_{j\geq 0}$ be the subsequence convergent to $ z^*\in \textrm{crit}(\xi)$. Thus, $(x^{k_j})_{j\geq 0}\rightarrow x^*\in \textrm{crit}(F)$. From the continuity condition,  $F(x^{k_j})\rightarrow F(x^*)$.   Thus, we have $(\xi(z^{k_j}))_{j\geq 0}\rightarrow \xi(z^*)$.  And then it holds
$$\xi(z^*)= \lim_j\xi(z^{k_j})=\lim_{k}\xi(x^k)=l.$$

(4') The proof is similar to (4).

\end{proof}

\begin{lemma}\label{conver}
Suppose that $F$ is a closed semi-algebraic function and coercive. Let the sequence $(x^k)_{k\geq 0}$
be generated by general scheme and the conditions (\ref{condition}) and (\ref{pointass})  hold.  If there exists $\theta>1$ such that the sequence $(\eta_k)_{k\geq 0}$ satisfies
\begin{equation}\label{totalcondition}
  \sum_{k}\eta_k<+\infty,\,~~ \textrm{and}~~\sum_{k}\left(\sum_{l=k}^{+\infty} \eta_l^2\right)^{\frac{\theta-1}{\theta}}<+\infty.
\end{equation}
Then, the sequence $(x^k)_{k\geq 0}$ has finite length, i.e.
\begin{equation}
\sum_{k=0}^{+\infty}\|x^{k+1}-x^k\|<+\infty.
\end{equation}
And $(x^k)_{k\geq 0}$ converges to a critical point $x^*$ of $F$.
\end{lemma}
\begin{proof}
Obviously, $\xi$ is semi-algebraic, and then K{\L}.
Let $x^*$ be a cluster point of $(x^k)_{k\geq 0}$, then, $z^*=(x^*,0)$ is also a cluster point of $(z^k)_{k\geq 0}$. If $\xi(z^{K'})=\xi(z^*)$ for some $K'$, with the fact $(\xi(z^k))_{k\geq 0}$ is decreasing, $\xi(z^k)=\xi(z^*)$ as $k\geq K'$. Using Lemma \ref{descend}, $z^k=z^{K'}$ as $k>K'$. In the following, we consider the case $\xi(z^k)>\xi(z^*)$.  From Lemmas \ref{con} and \ref{points}, there exist $\delta,\varepsilon>0$ such that for any $\overline{x}\in \mathcal{M}(z^0)$ and any $x$ satisfying that $\textrm{dist}(z,\mathcal{M}(z^0))<\varepsilon$ and $\xi(z^*)<\xi(z)<\xi(z^*)+\delta$. From Lemma \ref{points}, as $k$ is large enough,
$$z^k\in \{z\mid \textrm{dist}(z,\mathcal{M}(z^0))<\varepsilon\}\bigcap\{z\mid \xi(z^*)<\xi(z)<\xi(z^*)+\delta\}.$$
Thus, there exist concave function $\varphi$ such that
\begin{equation}
\varphi'(\xi(z^{k+1})-\xi(z^*))\cdot \textrm{dist}(\textbf{0},\partial\xi(z^{k+1}))\geq 1.
\end{equation}
Therefore, we have
\begin{align}
&\varphi(\xi(z^{k+1})-\xi(z^*))-\varphi(\xi(z^{k+2})-\xi(z^*))\nonumber\\
&\quad\quad\overset{a)}{\geq} \varphi'(\xi(z^{k+1})-\xi(z^*))\cdot(\xi(z^{k+1})-\xi(z^{k+2}))\nonumber\\
&\quad\quad\overset{b)}{\geq} a\cdot\varphi'(f(x^{k+1})-\xi(z^*))\cdot\|\omega^{k+2}-\omega^{k+1}\|^2\nonumber\\
&\quad\quad\overset{c)}{\geq} \frac{a\|\omega^{k+2}-\omega^{k+1}\|^2}{\textrm{dist}(\textbf{0},\partial \xi(z^{k+1}))}\overset{d)}{\geq} \frac{a\|\omega^{k+2}-\omega^{k+1}\|^2}{c\|x^{k+1}-x^k\|+d\eta_k+(t_{k+1})^{\theta-1}},\nonumber
\end{align}
where $a)$ is due to the concavity of $\varphi$, and $b)$ depends on Lemma \ref{descend}, $c)$ uses the K{\L} property, and $d)$ follows from Lemma \ref{lemma3}.
That is also
\begin{align}\label{th1+t1}
&2\|\omega^{k+2}-\omega^{k+1}\|\nonumber\\
&\quad\leq\frac{2}{a}\left\{[\varphi(\xi(z^{k+1})-\xi(z^*))-\varphi(\xi(z^{k+2})-\xi(z^*))]\cdot[c\sum_{j=k-\tau}^k\|\omega^{j+1}-\omega^j\|
+d\eta_k+(t_{k+1})^{\theta-1}]\right\}^{\frac{1}{2}}\nonumber\\
&\quad\overset{e)}{\leq}\frac{c(\tau+1)}{a^2}[\varphi(\xi(z^{k+1})-\xi(z^*))-\varphi(\xi(z^{k+2})-\xi(z^*))]\nonumber\\
&\quad+\frac{\sum_{j=k-\tau}^k\|\omega^{j+1}-\omega^j\|}{\tau+1}+\frac{ad}{c(\tau+1)}\eta_k+\frac{a}{c(\tau+1)}(t_{k+1})^{\theta-1},
\end{align}
where $e)$ uses the  Schwarz  inequality $2(xy)^{\frac{1}{2}}\leq tx+\frac{t}{y}$ with $x=[\varphi(\xi(z^{k+1})-\xi(z^*))-\varphi(\xi(z^{k+2})-\xi(z^*))]$, and $y=[c\sum_{j=k-\tau}^k\|\omega^{j+1}-\omega^j\|+d\eta_k+(t_{k+1})^{\theta-1}]$, and $t=\frac{c(\tau+1)}{a}$.
Multiplying \eqref{th1+t1} with $\tau+1$, we have
\begin{eqnarray}
2(\tau+1)\|\omega^{k+2}-\omega^{k+1}\|&\leq&\frac{c(\tau+1)^2}{a^2}[\varphi(\xi(z^{k+1})-\xi(z^*))-\varphi(\xi(z^{k+2})-\xi(z^*))]\nonumber\\
&+&\sum_{j=k-\tau}^k\|\omega^{j+1}-\omega^j\|+\frac{ad}{c}\eta_k+\frac{a}{c}(t_{k+1})^{\theta-1}.
\end{eqnarray}
Summing both sides from $k$ to $K$, and with simplifications,
\begin{align}
&(2\tau+1)\sum_{l=k+1}^{K+1}\|\omega^{l+1}-\omega^l\|+(2\tau+2)\sum_{j=K+1-\tau}^{K+1}\|\omega^{j+1}-\omega^{j}\|\nonumber\\
&\quad\quad\leq\frac{c(\tau+1)^2}{a^2}[\varphi(\xi(z^{k+1})-\xi(z^*))-\varphi(\xi(z^{K+2})-\xi(z^*))]\nonumber\\
&\quad\quad+\sum_{j=k-\tau}^k\|\omega^{j+1}-\omega^j\|+\frac{ad}{c}\sum_{l=k}^{K}\eta_l+\frac{a}{c}\sum_{l=k+1}^{K+1}(t_{l})^{\theta-1}<+\infty.
\end{align}
Letting $K\rightarrow+\infty$ and using $\sum_{j=K+1-\tau}^{K+1}\|\omega^{j+1}-\omega^{j}\|\rightarrow0$ and $\tau\in \mathbb{Z}^0+$, we then derive
\begin{align}
\sum_{k}\|\omega^{k+1}-\omega^{k}\|<+\infty.
\end{align}
By using (\ref{pointass}), we are then led to
\begin{align}
\sum_{k}\|x^{k+1}-x^{k}\|<+\infty.
\end{align}
Thus, $(x^k)_{k\geq 0}$ has only one stationary point $x^*$. From Lemma \ref{points}, $x^*\in \textrm{crit}(F)$.
\end{proof}

The requirement \eqref{totalcondition} is complicated and impractical in the applications. Thus, we consider the sequence $(\eta_k)_{k\geq 0}$ enjoys the polynomial forms as $\eta_k\leq\frac{C}{k^{\alpha}}$ with $\alpha>1$ and some $C>0$. We try to simplify \eqref{totalcondition} in this case.
The task  then reduces to the following mathematical analysis problem: find the minimum $\alpha_0\geq 1$ such that  for any $\alpha\in (\alpha_0,+\infty)$, there exists $\theta>1$ that makes \eqref{totalcondition} hold. Direct calculations give us
\begin{align}
\left(\sum_{l=k}^{+\infty} \eta_l^2\right)^{\frac{\theta-1}{\theta}}\leq\left(\sum_{l=k}^{+\infty} \frac{C^2}{l^{2\alpha}}\right)^{\frac{\theta-1}{\theta}}\leq\left(\sum_{l=k-1}^{+\infty} \int_{l}^{l+1}\frac{C^2}{t^{2\alpha}}dt\right)^{\frac{\theta-1}{\theta}}=\frac{C^{\frac{2(\theta-1)}{\theta}}}{(2\alpha-1)^{\frac{\theta-1}{\theta}}}
\cdot\frac{1}{(k-1)^{\frac{(2\alpha-1)(\theta-1)}{\theta}}}.
\end{align}
Thus, we need
\begin{align}
\alpha>1,~~\textrm{and}~~\frac{(2\alpha-1)(\theta-1)}{\theta}>1.
\end{align}
After simplifications, we get
\begin{align}
\alpha>1,~~\textrm{and}~~\alpha>\frac{2\theta-1}{2(\theta-1)}.
\end{align}
Then, the problem reduces to
\begin{align}
\alpha_0=\inf_{\theta>1}\left\{c(\theta):=\max\{1,\,\frac{2\theta-1}{2(\theta-1)}\}=\frac{2\theta-1}{2(\theta-1)}\right\}.
\end{align}
Figure \ref{fig.1} shows the function values between $[1.1,5]$.   It is easy to verify that $c(\theta)$ is decreasing to $1$ at $+\infty$. Therefore, we get $\alpha_0=1$. That is also to say if $\eta_k\leq\frac{C}{k^{\alpha}}$ with any fixed $\alpha>1$, there exists $\theta>1$ such that \eqref{totalcondition} can hold. And then, the sequence $(x^k)_{k\geq 0}$ is convergent to some critical point of $F$. Therefore, we obtain the following result.

\begin{theorem}[Convergence result]\label{Th-conver}
Suppose that function $F$ is  closed,  semi-algebraic   and coercive. Let  conditions (\ref{condition}) and (\ref{pointass})  hold, and the sequence $(\eta_k)_{k\geq 0}$ obey
\begin{align}\label{assass}
\eta_k=\mathcal{O}(\frac{1}{k^{\alpha}}),\,\alpha>1.
\end{align}
Then, the sequence $\{x^k\}_{k=0,1,2,3,\ldots}$ has finite length, i.e.
\begin{equation}
\sum_{k=0}^{+\infty}\|x^{k+1}-x^k\|<+\infty.
\end{equation}
And $\{x^k\}_{k=0,1,2,3,\ldots}$ converges to a critical point $x^*$ of $F$.
\end{theorem}
\begin{figure}
  \centering
  \includegraphics[width=3.0in]{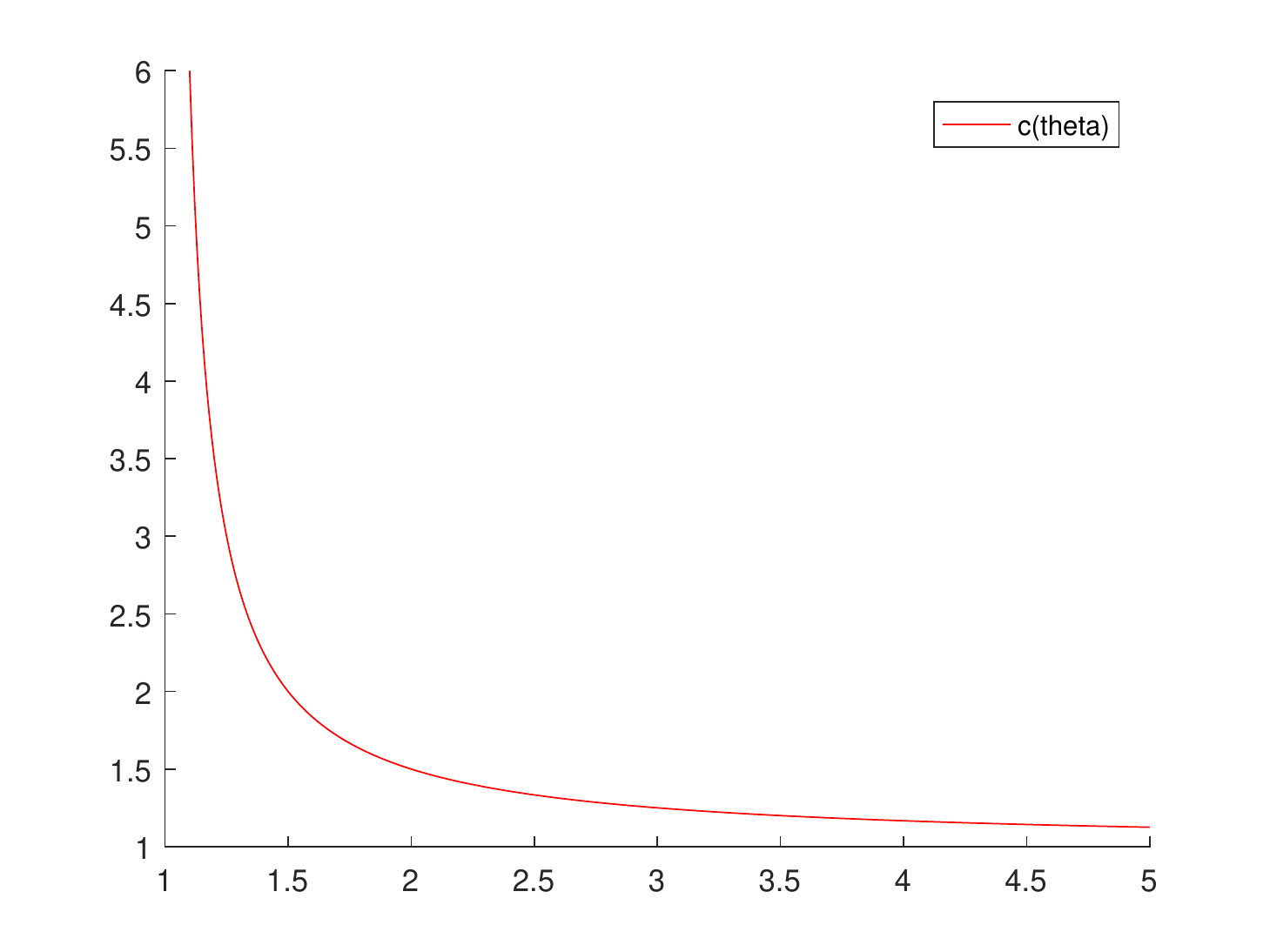}\\
  \caption{Function $c(\theta)$  on the interval $[1.1, 5]$ }\label{fig.1}
\end{figure}

\section{Applications to several nonconvex algorithms}
In this part,  several classical nonconvex inexact algorithms are considered. We apply our theoretical findings to these algorithms and derive corresponding convergence results. As presented before, we just need to check whether the algorithm satisfies   conditions  (\ref{condition}) and   \eqref{pointass}.  For a closed function (may be nonconvex) $J$, we denote
\begin{equation}\label{pro}
    \textbf{prox}_{J}(x)\in \textrm{arg}\min_{y}\{ J(y)+\frac{\|y-x\|^2}{2}\}.
\end{equation}
Different from convex cases, the $\textbf{prox}_{J}$ is a point-to-set operator and may enjoy more than one solution.
The definition of the proximal map directly gives the following result.
\begin{lemma}\label{tool}
For any $x$ and $y$, if $z\in \emph{\textbf{prox}}_{J}(x)$,
\begin{equation}
    J(z)+\frac{\|z-x\|^2}{2}\leq J(y)+\frac{\|y-x\|^2}{2}.
\end{equation}
\end{lemma}
Of course, we also have
\begin{equation}
    x-z\in\partial J(z).
\end{equation}
\\
In subsections 4.1-4.4, the point $\omega^k$ is $x^k$ itself, i.e., $\omega^k\equiv x^k$.
\subsection{Inexact nonconvex    proximal gradient algorithm}
The nonconvex proximal gradient algorithm is developed for the nonconvex composite optimization
\begin{equation}\label{composite}
\min_{x}\{F(x)=f(x)+g(x)\},
\end{equation}
where $f$ is differentiable and $\nabla f$ is Lipschitz with $L$, and $g$ is closed. And both $f$ and $g$ may be nonconvex.
The nonconvex inexact proximal gradient algorithm  can be described as
\begin{equation}\label{pga}
    x^{k+1}=\textbf{prox}_{h g}(x^k-h\nabla f(x^k)+e^k),
\end{equation}
where $h$ is the stepsize, $\textbf{prox}$ is the proximal operator and $e^k$ is the noise. In the convex case, this algorithm is discussed in \cite{villa2013accelerated,schmidt2011convergence}, and the acceleration is studied in \cite{sun2017convergencein}.
\begin{lemma}\label{pgalemma1}
Let $0<h<\frac{1}{L}$ and the sequence $(x^k)_{k\geq 0}$ be generated by algorithm (\ref{pga}), we have
\begin{equation}
    F(x^k)-F(x^{k+1})\geq \frac{1}{4}(\frac{1}{h}-L)\|x^{k+1}-x^k\|^2-\frac{1}{h(1-hL)}\|e^k\|^2.
\end{equation}
\end{lemma}
\begin{proof}
The $L$-Lipschitz of $\nabla f$ gives
\begin{equation}\label{pgalemma1-t1}
    f(x^{k+1})-f(x^k)\leq \langle\nabla f(x^k),x^{k+1}-x^k\rangle+\frac{L}{2}\|x^{k+1}-x^k\|^2.
\end{equation}
On the other hand,
with Lemma \ref{tool},  we have
\begin{equation}
    h g(x^{k+1})+\frac{\|x^k-h\nabla f(x^k)+e^k-x^{k+1}\|^2}{2}\leq h g(x^k)+\frac{\|-h\nabla f(x^k)+e^k\|^2}{2}.
\end{equation}
This is also
\begin{equation}\label{pgalemma1-t2}
    g(x^{k+1})-g(x^k)\leq -\langle \nabla f(x^k),x^{k+1}-x^k\rangle-\frac{\|x^k-x^{k+1}\|^2}{2h}+\frac{\langle e^k,x^{k+1}-x^k\rangle}{h}.
\end{equation}
Summing (\ref{pgalemma1-t1}) and (\ref{pgalemma1-t2}),
\begin{equation}\label{pgalemma1-t3}
    F(x^{k+1})-F(x^k)\leq\frac{1}{2}(L-\frac{1}{h})\|x^{k+1}-x^k\|^2+\frac{\langle e^k,x^{k+1}-x^k\rangle}{h}.
\end{equation}
With the Cauchy-Schwarz inequality, we have
\begin{equation}\label{pgalemma1-t4}
    \frac{\langle e^k,x^{k+1}-x^k\rangle}{h}\leq \frac{1}{4}(\frac{1}{h}-L)\|x^{k+1}-x^k\|^2+\frac{1}{h(1-hL)}\|e^k\|^2
\end{equation}
Combining (\ref{pgalemma1-t4}) and (\ref{pgalemma1-t3}), we then prove the result.
\end{proof}

\begin{lemma}\label{proxerrorbound}
Let the sequence $(x^k)_{k\geq 0}$ be generated by algorithm (\ref{pga}), we have
\begin{equation}
   \emph{dist}(\textbf{0},\partial F(x^{k+1}))\leq (\frac{1}{h}+L)\|x^k-x^{k+1}\|+\frac{1}{h}\|e^k\|.
\end{equation}
\end{lemma}
\begin{proof}
We have
\begin{equation}
   \frac{x^k-x^{k+1}}{h}-\nabla f(x^k)+\frac{e^k}{h}\in \partial g(x^{k+1}).
\end{equation}
Therefore,
\begin{equation}
   \frac{x^k-x^{k+1}}{h}+\nabla f(x^{k+1})-\nabla f(x^k)+\frac{e^k}{h}\in \nabla f(x^{k+1})+\partial g(x^{k+1})=\partial F(x^{k+1}).
\end{equation}
Thus, we have
\begin{eqnarray}
      \textrm{dist}(\textbf{0},\partial F(x^{k+1}))&\leq& \|\frac{x^k-x^{k+1}}{h}+\nabla f(x^{k+1})-\nabla f(x^k)+\frac{e^k}{h}\|\nonumber\\
      &\leq&\frac{1}{h}\|x^k-x^{k+1}\|+L\|x^k-x^{k+1}\|+\frac{\|e^k\|}{h}.
\end{eqnarray}
\end{proof}

\begin{lemma}
Let $0<h<\frac{1}{L}$ and the sequence $(x^k)_{k\geq 0}$ be generated by algorithm (\ref{pga}), and $F$ be coercive. We also assume that $e^k\rightarrow \textbf{0}$. Then, for  $x^*$ being the stationary point of $(x^k)_{k\geq 0}$, there exists a subsequence $(x^{k_j})_{j\geq 0}$ converges to $x^*$ satisfying $F(x^{k_j})\rightarrow F(x^*)$ and $x^*\in \emph{crit}(F)$.
\end{lemma}
\begin{proof}
With Lemma \ref{pgalemma1}, $(x^k)_{k\geq 0}$ is bounded. For any $x^*\in\textrm{crit}(F)$, there exists a subsequence $(x^{k_j})_{j\geq 0}$ converges to $x^*$. With Lemmas \ref{points} and \ref{pgalemma1}, we also have
\begin{equation}
    x^{k_j-1}\rightarrow x^*.
\end{equation}
And in each iteration, with Lemma \ref{tool}, we have
\begin{equation}
    h g(x^{k_j})+\frac{\|x^{k_j-1}-h\nabla f(x^{k_j-1})+e^{k_j}-x^{k_j}\|^2}{2}\leq h g(x^{*})+\frac{\|x^{k_j-1}-h\nabla f(x^{k_j-1})+e^{k_j-1}-x^*\|^2}{2}.
\end{equation}
Taking $j\rightarrow+\infty$, we have
\begin{equation}
    \underset{j\rightarrow+\infty}{\lim\sup}~~g(x^{k_j})\leq g(x^*).
\end{equation}
And recalling the lower semi-continuity of $g$,
\begin{equation}
     g(x^*)\leq \underset{j\rightarrow+\infty}{\lim\inf}~~g(x^{k_j}).
\end{equation}
That means $\lim_{j} g(x^{k_j})=g(x^*)$; and combining the continuity of $f$, we then prove $\lim_{j} F(x^{k_j})=F(x^*)$.
With Lemma \ref{proxerrorbound}, we can see $\textrm{dist}(\textbf{0},\partial F(x^*))=0$, that is, $x^*\in \textrm{crit}(F)$.
\end{proof}
With previous results, we then prove the following proposition.
\begin{proposition}\label{converpga}
Suppose that $f$ and $g$ are both semi-algebraic,  $F$ is  coercive, and $0<h<\frac{1}{L}$. Let the sequence $(x^k)_{k\geq 0}$
be generated by  scheme (\ref{pga}). If the sequence $(e^k)_{k\geq 0}$ satisfies
\begin{equation}\label{pgacond}
   \|e^k\|=\mathcal{O}(\frac{1}{k^{\alpha}}),\,\alpha>1.
\end{equation}
Then, the sequence $(x^k)_{k\geq 0}$ has finite length, i.e.
\begin{equation}
\sum_{k=0}^{+\infty}\|x^{k+1}-x^k\|<+\infty.
\end{equation}
And $\{x^k\}_{k=0,1,2,3,\ldots}$ converges to a critical point $x^*$ of $F$.
\end{proposition}
\begin{proof}
From (\ref{pgacond}), we have $\|e^k\|\rightarrow 0$. And $F$ is a semi-algebraic function. With lemmas proved before in this subsection and Theorem \ref{Th-conver}, we then obtain the result.
\end{proof}
\subsection{Inexact proximal alternating linearized  minimization algorithm}
In this part, we use the following convention
$$x=(y,z), x^k=(y^k,z^k), e^k=(\alpha^k, \beta^k)$$
The following problem  is considered
\begin{equation}\label{al}
\min_{y,z}\{\Phi(y,z):=f(y)+H(y,z)+g(z)\},
\end{equation}
where the function $H$ is assumed to be differentiable and satisfy
\begin{subequations}
\begin{align}
&\|\nabla_y H(y^1,z)-\nabla_y H(y^2,z)\|\leq M\|y^1-y^2\|,\\
&\|\nabla_z H(y,z^1)-\nabla_z H(y,z^2)\|\leq N\|z^1-z^2\|,\\
&\|\nabla_y H(x^1)-\nabla_y H(x^2)\|\leq L\|x^1-x^2\|.
\end{align}
\end{subequations}
An intuitive algorithm for solving problem (\ref{al}) is the  alternating minimization scheme, i.e.,  fixing one of $y$ and $z$ in each iteration and then minimizing the other one \cite{nocedal2006sequential}; and the convergence rate is proved in \cite{beck2015convergence} in the convex case. In the nonconvex case, the alternating minimization scheme can barely derive the descent property, thus the authors propose the proximal alternating minimization \cite{attouch2010proximal}. However, both  alternating minimization  and proximal  alternating minimization have an obvious drawback:  a minimization problem is solved in each
iteration, the stopping criterion may be hard to determine, and error accumulates. Therefore, several variants are developed \cite{bolte2014proximal,sun2017little,shefi2016rate}, and the Proximal Alternating  Linearized  Minimization (PALM) algorithm \cite{bolte2014proximal} is one of them. The inexact PALM can be described as
\begin{subequations}\label{plam}
\begin{align}
&y^{k+1}=\textbf{prox}_{\gamma f}(y^k-\gamma\nabla_{y}H(y^k,z^k)+\alpha^k),\\
&z^{k+1}=\textbf{prox}_{\lambda g}(z^k-\lambda\nabla_{z}H(y^{k+1},z^k)+\beta^k).
\end{align}
\end{subequations}

\begin{lemma}\label{plamlemma1}
Let  the sequence $(x^k)_{k\geq 0}$ be generated by algorithm (\ref{plam}). If
\begin{equation}\label{plamcondition}
    \min\{M-\frac{1}{\gamma},\,~N-\frac{1}{\lambda}\}>0,
\end{equation}
we have
\begin{equation}
    \Phi(x^k)-\Phi(x^{k+1})\geq \nu\|x^{k+1}-x^k\|^2-\sigma\|e^k\|^2,
\end{equation}
where $\nu:=\min\{\frac{1}{4}(M-\frac{1}{\gamma}),\frac{1}{4}(N-\frac{1}{\lambda})\}$ and $\sigma:=\max\{\frac{1}{\gamma(1-\gamma M)}, \frac{1}{\lambda_k(1-\lambda N)}\}$
\end{lemma}
\begin{proof}
The $M$-Lipschitz of $\nabla_{y} H(y,z^k)$ gives
\begin{equation}\label{plamlemma1-t1}
    H(y^{k+1},z^k)-H(y^k,z^k)\leq \langle\nabla_y H(y^k,z^k),y^{k+1}-y^k\rangle+\frac{M}{2}\|y^{k+1}-y^k\|^2.
\end{equation}
From Lemma \ref{tool},  we have
\begin{equation}
    \gamma f(y^{k+1})+\frac{\|y^k-\gamma\nabla_{y}H(y^k,z^k)+\alpha^k-y^{k+1}\|^2}{2}\leq \gamma f(y^k)+\frac{\|-\gamma\nabla_{y}H(y^k,z^k)+\alpha^k\|^2}{2}.
\end{equation}
This is also
\begin{equation}\label{plamlemma1-t2}
    f(y^{k+1})-f(y^k)\leq -\langle \nabla_y H(y^k,z^k),y^{k+1}-y^k\rangle-\frac{\|y^k-y^{k+1}\|^2}{2\gamma}+\frac{\langle \alpha^k,y^{k+1}-y^k\rangle}{\gamma}.
\end{equation}
Summing (\ref{plamlemma1-t1}) and (\ref{plamlemma1-t2}), with the Cauchy-Schwarz inequality
\begin{equation}\label{plamlemma1-t4}
    \frac{\langle \alpha^k,y^{k+1}-y^k\rangle}{\gamma}\leq \frac{1}{4}(\frac{1}{\gamma}-M)\|y^{k+1}-y^k\|^2+\frac{1}{\gamma(1-\gamma M)}\|\alpha^k\|^2,
\end{equation}
we then have
\begin{equation}\label{plamlemma1-t3}
    [f(y^{k+1})+H(y^{k+1},z^k)]-[f(y^k)+H(y^{k},z^k)]\leq\frac{1}{4}(M-\frac{1}{\gamma})\|y^{k+1}-y^k\|^2+\frac{\|\alpha^k\|^2}{\gamma(1-\gamma M)}.
\end{equation}
Similarly, we can prove
\begin{equation}\label{plamlemma1-t5}
    [g(z^{k+1})+H(y^{k+1},z^{k+1})]-[g(z^k)+H(y^{k+1},z^k)]\leq\frac{1}{4}(N-\frac{1}{\lambda})\|z^{k+1}-z^k\|^2+\frac{\|\beta^k\|^2}{\lambda(1-\lambda N)}.
\end{equation}
Combining (\ref{pgalemma1-t4}) and (\ref{pgalemma1-t3}), we then prove the result.
\end{proof}

\begin{lemma}\label{plamre}
Let the sequence be generated by algorithm (\ref{plam}),
 we have
\begin{equation}
   \emph{dist}(\textbf{0},\partial \Phi(x^{k+1}))\leq S\|x^k-x^{k+1}\|+D\|e^k\|,
\end{equation}
where $S:=\frac{1}{\lambda}+\frac{1}{\gamma}+2L$ and $D:=\sqrt{\frac{1}{\gamma^2}+\frac{1}{\lambda^2}}$.
\end{lemma}
\begin{proof}
In updating $y^{k+1}$, we have
\begin{equation}
   \frac{y^k-y^{k+1}}{\gamma}-\nabla_y H(y^k,z^k)+\frac{\alpha^k}{\gamma}\in \partial f(y^{k+1}).
\end{equation}
Therefore, we are then led to
\begin{equation}
    \frac{y^k-y^{k+1}}{\gamma}+\nabla_y H(y^{k+1},z^{k+1})-\nabla_y H(y^k,z^k)+\frac{\alpha^k}{\gamma}\in \nabla_y H(y^{k+1},z^{k+1})+\partial f(y^{k+1})=\partial_y \Phi(x^{k+1}).
\end{equation}
Thus, we have
\begin{eqnarray}\label{plamre-t1}
      \textrm{dist}(\textbf{0},\partial_y \Phi(x^{k+1}))&\leq& \|\frac{y^k-y^{k+1}}{\gamma}+\nabla_y H(y^{k+1},z^{k+1})-\nabla_y H(y^k,z^k)+\frac{\alpha^k}{\gamma}\|\nonumber\\
      &\leq&\frac{\|y^k-y^{k+1}\|}{\gamma}+L\|x^{k+1}-x^k\|+\frac{\|\alpha^k\|}{\gamma}.
\end{eqnarray}
In updating $z^{k+1}$, we have
\begin{eqnarray}\label{plamre-t2}
      \textrm{dist}(\textbf{0},\partial_z \Phi(x^{k+1}))
      \leq\frac{\|z^k-z^{k+1}\|}{\lambda}+L\|z^{k+1}-z^k\|+\frac{\|\beta^k\|}{\lambda}.
\end{eqnarray}
Combining (\ref{plamre-t1}) and (\ref{plamre-t2}), we then prove the result.
\end{proof}

\begin{lemma}
Let  the sequence $(x^k)_{k\geq 0}$ be generated by algorithm (\ref{pga}), and $\Phi$ be coercive, and condition (\ref{plamcondition}) hold, $e^k\rightarrow \textbf{0}$.  Then, for any $x^*$ being the stationary point of $(x^k)_{k\geq 0}$, there exists a subsequence $(x^{k_j})_{j\geq 0}$ converges to $x^*$ satisfying $\Phi(x^{k_j})\rightarrow \Phi(x^*)$ and $x^*\in \emph{crit}(\Phi)$.
\end{lemma}
\begin{proof}
With Lemma \ref{plamlemma1}, $(x^k)_{k\geq 0}$ is bounded. For any $x^*\in\textrm{crit}(\Phi)$, there exists a subsequence $(x^{k_j})_{j\geq 0}$ converges to $x^*$. With Lemmas \ref{points} and \ref{plamlemma1}, we also have
\begin{equation}
    x^{k_j-1}=(y^{k_j-1},z^{k_j-1})\rightarrow x^*=(y^*,z^*).
\end{equation}
And in each iteration of updating $y^{k_j}$, with Lemma \ref{tool}, we have
\begin{eqnarray}
    \gamma f(y^{k_j})&+&\frac{\|y^{k_j-1}-\gamma\nabla_{y}H(y^{k_j-1},z^{k_j-1})+\alpha^{k_j-1}-y^{k_j}\|^2}{2}\nonumber\\
    &\leq& \gamma f(y^{*})+\frac{\|y^{k_j-1}-\gamma\nabla_{y}H(y^{k_j-1},z^{k_j-1})+\alpha^{k_j-1}-y^*\|^2}{2}.
\end{eqnarray}
Taking $j\rightarrow+\infty$, we have
\begin{equation}
    \underset{j\rightarrow+\infty}{\lim\sup}~~f(y^{k_j})\leq f(y^*).
\end{equation}
And recalling the lower semi-continuity of $f$,
\begin{equation}
     f(y^*)\leq \underset{j\rightarrow+\infty}{\lim\inf}~~f(y^{k_j}).
\end{equation}
That means $\lim f(y^{k_j})=f(x^*)$; and similarly, $\lim g(z^{k_j})=g(z^*)$;  combining the continuity of $H$, we then prove the result.
\end{proof}

And then, we then prove the following result.

\begin{proposition}
Suppose that $\Phi$ is   coercive, and condition \eqref{plamcondition}  holds. Functions $f$, $g$ and $H$ are all semi-algebraic. Let the sequence $(x^k)_{k\geq 0}$
be generated by  scheme (\ref{plam}). If the sequence $(\alpha^k,\beta^k)_{k\geq 0}$ satisfies
\begin{equation}
   \|\alpha^l\|+\|\beta^l\|=\mathcal{O}(\frac{1}{k^{\alpha}}),\,\alpha>1.
\end{equation}
Then, the sequence $(x^k)_{k\geq 0}$ has finite length, i.e.
\begin{equation}
\sum_{k=0}^{+\infty}\|x^{k+1}-x^k\|<+\infty.
\end{equation}
And $\{x^k\}_{k=0,1,2,3,\ldots}$ converges to a critical point $x^*$ of $\Phi$.
\end{proposition}
\subsection{Inexact proximal reweighted algorithm}
This part  considers an iteratively reweighted algorithm for a broad class of nonconvex and nonsmooth problems with the following form
\begin{equation}\label{rem}
    \min_{x}\{\Psi(x)=f(x)+\sum_{i=1}^{N}h(g(x_i))\},
\end{equation}
where $x\in \mathbb{R}^N$, and function $f$ has a Lipschitz gradient with constant $L_f$, and
$g:\mathbb{R}\rightarrow \mathbb{R}$ is a lower-semicontinuous convex function, and $h:\textrm{Im}(g)\rightarrow \mathbb{R}$ is a differentiable concave function  with a Lipschitz continuous gradient with constant $L_{h}$, i.e.,
\begin{equation}
    \mid h'(s)-h'(t)\mid\leq L_{h}\mid s-t\mid,
\end{equation}
and $h'(t)>0$ for any $t\in \textrm{Im}(g)$. This model generalizes various problems in the machine learning and signal processing satisfy. The reweighted style algorithms \cite{chartrand2008iteratively,candes2008enhancing,daubechies2010iteratively,lai2013improved,sun2017convergence,lu2017ell,chartrand2016nonconvex} (or also called multi-stage algorithm \cite{zhang2010analysis}) are popular in solving this problem. To make each subproblem easy to be solved. The Proximal Iteratively REweighted (PIRE) algorithm is proposed in \cite{lu2014proximal}. The convergence of PIRE under K{\L} property is proved by \cite{sun2017global}.  We consider the inexact version of PIRE as
\begin{equation}\label{reweight}
    x^{k+1}_i=\textbf{prox}_{\mu w^k_i g}(x^k_i-\mu\nabla_i f(x^k)+e_i^k), i\in[1,2,\ldots,N]
\end{equation}
where $w^k_i:=h'(g(x^k_i))$ and $\mu>0$ is the stepsize, $e^k$ is the noise vector. If $e^k\equiv\textbf{0}$, the algorithm then reduces to PIRE.

\begin{lemma}\label{rewlemma1}
Let $(x^{k})_{k\geq 0}$ be generated by scheme (\ref{reweight}) and $0<\mu<\frac{2}{L_f}$.
Then, we  have
\begin{equation}
    \Psi(x^k)-\Psi(x^{k+1})\geq (\frac{1}{\mu}-\frac{L_{f}}{2})\|x^{k}-x^{k+1}\|^2-\frac{\|e^k\|^2}{\mu(2-\mu L_f)}.
\end{equation}
\end{lemma}
\begin{proof}
We can easily obtain that
\begin{eqnarray}\label{rewlema1temp1}
\Psi(x^k)-\Psi(x^{k+1})&=&f(x^k)-f(x^{k+1})+\sum_{i=1}^N h(g(x^k_i))-h(g(x^{k+1}_i))\nonumber\\
&\geq& \langle \nabla f(x^{k}),x^k-x^{k+1}\rangle-\frac{L_f}{2}\|x^k-x^{k+1}\|_2^2+\sum_{i=1}^N h(g(x^k_i))-h(g(x^{k+1}_i))\nonumber\\
&\geq&\sum_{i=1}^N \langle \nabla_i f(x^{k}),x^k_i-x^{k+1}_i\rangle-\frac{L_f}{2}\|x^k-x^{k+1}\|^2\nonumber\\
&+&\sum_{i=1}^N w^k_i(g(x^k_i)-g(x^{k+1}_i)).
\end{eqnarray}
Note that $x^{k+1}_{i}$ is obtained by (\ref{reweight}); the K.K.T condition gives
\begin{equation}\label{rewlema1temp2}
    \nabla_{i} f(x^k)+w^{k}_{i}v^{k+1}_{i}+\frac{(x^{k+1}_{i}-x^{k}_{i})}{\mu}-\frac{e^k_i}{\mu}=\textbf{0},
\end{equation}
where $v^{k+1}_{i}\in \partial g(x^{k+1}_i)$.
 Noting that $g$ is convex and $w^k_i>0$,  we have
\begin{equation}\label{rewlema1temp3}
    \sum_{i=1}^N w^k_i(g(x^k_i)-g(x^{k+1}_i))\geq \sum_{i=1}^N \langle w^{k}_{i}v^{k+1}_{i},x^k_{i}-x^{k+1}_{i}\rangle.
\end{equation}
Substituting (\ref{rewlema1temp2}) and (\ref{rewlema1temp3}) into (\ref{rewlema1temp1}), we derive that
\begin{eqnarray}\label{rewlema1temp4}
\Psi(x^k)-\Psi(x^{k+1})&\geq&(\frac{1}{\mu}-\frac{L_f}{2})\|x^k-x^{k+1}\|^2+\frac{\langle e^k,x^k-x^{k+1}\rangle}{\mu}\nonumber\\
&\geq&\frac{1}{2}(\frac{1}{\mu}-\frac{L_f}{2})\|x^k-x^{k+1}\|^2-\frac{\|e^k\|_2^2}{\mu(2-\mu L_f)},
\end{eqnarray}
where we use the  inequality $\langle e^k,x^k-x^{k+1}\rangle\geq -\frac{1}{2}(1-\frac{\mu L_f}{2})\|x^k-x^{k+1}\|^2-\frac{\|e^k\|_2^2}{2-\mu L_f}$.

\end{proof}

\begin{lemma}\label{rewlemma2}
Let $(x^{k})_{k\geq 0}$ be generated by scheme (\ref{reweight}) and $0<\mu<\frac{2}{L_f}$, and function $\Phi$ be coercive. Then, there exist $S,D>0$ such that
\begin{equation}
  \emph{dist}(\textbf{0}, \partial\Psi(x^{k+1}))\leq S\|x^{k+1}-x^{k}\|+D\|e^k\|.
\end{equation}
\end{lemma}
\begin{proof}
We can easily have
\begin{equation}\label{rewlemma2t1}
    \nabla f(x^{k+1})+W^{k+1}v^{k+1}\in\partial \Psi(x^{k+1}),
\end{equation}
where $v^{k+1}_i\in \partial g(x^{k+1}_i)$ and $W^{k+1}=\textrm{diag}(h'(g(x^k_1)),h'(g(x^k_2)),\ldots,h'(g(x^k_N)))$. Recall relation (\ref{rewlema1temp2}), we have
\begin{equation}\label{rewlemma2t2}
  v^{k+1}_{i}=[\frac{e^k_i}{\mu}+\frac{(x^{k}_{i}-x^{k+1}_{i})}{\mu}-\nabla_{i} f(x^k)]/w^{k}_{i}, i \in[1,2,\ldots,N].
\end{equation}
Combining (\ref{rewlemma2t1}) and (\ref{rewlemma2t2}), we have
\begin{eqnarray}\label{rewlemma2t0}
    \frac{w^{k+1}_{i}-w^{k}_{i}}{w^{k}_{i}}\nabla_i f(x^{k+1})+\frac{w^{k+1}_{i}}{w^{k}_{i}}[\frac{e^k_i}{\mu}+\frac{(x^{k}_{i}-x^{k+1}_{i})}{\mu}]\in \partial_i \Psi(x^{k+1}), i \in[1,2,\ldots,N].
\end{eqnarray}
Due to that $\nabla f$ is continuous, so is $\nabla_{i} f(x)$; and from Lemmas \ref{rewlemma1} and \ref{points}, $\{x^k\}_{k=0,1,2,\ldots}$ is bounded. Hence, there exist $\widetilde{L}>0$ such that
\begin{equation}
    \max_{1\leq i\leq N} \|\nabla_i f(x^k)\|\leq \widetilde{L},\,~\forall~~k.
\end{equation}
Considering that $h'$ is nonzero and continuous, and $\{g(x^k_i)\}_{k=0,1,2,\ldots}$ is bounded ($i \in[1,2,\ldots,N]$).
 Therefore, for any $k$ and $i\in\{1,2,\ldots,N\}$, there exists $\delta,\pi>0$ such that
$$\delta\leq h'(g(x^k_i))=w^{k}_{i}\leq \pi.$$
With [Theorem 10.4, \cite{rockafellar2015convex}] and the convexity of $g$, there exists $d_g$
\begin{equation}
    |g(x^{k}_i)-g(x^{k+1}_i)|\leq d_g|x^{k}_i-x^{k+1}_i|.
\end{equation}
Hence, we derive
\begin{equation}\label{rewlemma2t3}
    \max_{1\leq i\leq N} \mid\frac{w^{k+1}_{i}}{w^{k}_{i}}\mid\leq\frac{\pi}{\delta}, \max_{1\leq i\leq N} \mid\frac{1}{w^{k}_{i}}\mid\leq\frac{1}{\delta}.
\end{equation}
From  (\ref{rewlemma2t0}), with (\ref{rewlemma2t3}), we have
\begin{eqnarray}\label{rewlemma2t4}
  \textrm{dist}(\textbf{0},\partial\Psi(x^{k+1}))&\leq&\sum_{i=1}^N \left|\frac{w^{k+1}_{i}-w^{k}_{i}}{w^{k}_{i}}\nabla_i f(x^{k+1})+\frac{w^{k+1}_{i}}{w^{k}_{i}}[\frac{e^k_i}{\mu}+\frac{(x^{k}_{i}-x^{k+1}_{i})}{\mu}]\right|\nonumber\\
  &\leq&\frac{\tilde{L}}{\delta}\sum_{i=1}^N|w^{k+1}_{i}-w^{k}_{i}|+\frac{\pi\sqrt{N}}{\mu\delta}\|e^k\|+\frac{\pi\sqrt{N}}{\mu\delta}\|x^{k+1}-x^k\|.
\end{eqnarray}
The problem also turns to estimating $|w^{k+1}_{i}-w^{k}_{i}|$. For any $i\in [1,2,\ldots,N]$,
\begin{eqnarray}\label{rewlemma2t5}
|w^{k}_{i}-w^{k+1}_{i}|&=&|h'(g(x^k_i))-h'(g(x^{k+1}_i))|\nonumber\\
&\leq& L_h |g(x^k_i)-g(x^{k+1}_i)|=L_h d_g |x^{k+1}_i-x^k_i|.
\end{eqnarray}
Combining (\ref{rewlemma2t4}) and (\ref{rewlemma2t5}), we obtain
\begin{equation}
  \textrm{dist}(\textbf{0},\partial\Psi(x^{k+1}))\leq \left(\frac{\tilde{L}L_h d_g \sqrt{N}}{\delta}+\frac{\pi\sqrt{N}}{\mu\delta}\right)\|x^{k+1}-x^k\|+\frac{\pi\sqrt{N}}{\mu\delta}\|e^k\|.
\end{equation}

\end{proof}

\begin{lemma}
Let $(x^{k})_{k\geq 0}$ be generated by scheme (\ref{reweight}), and function $\Psi$ be coercive. Then, for any $x^*$ being the stationary point of $(x^{k})_{k\geq 0}$, there exists a subsequence $(x^{k_j})_{j\geq 0}$ convergent to $x^*$ satisfying $\Psi(x^{k_j})\rightarrow \Psi(x^*)$ and $x^*\in \emph{crit}(\Psi)$.
\end{lemma}
\begin{proof}
The continuity of the function $\Psi$ directly gives the result.
\end{proof}

\begin{proposition}
Suppose that $f$, $g$, $h$ are all semi-algebraic,   and $\Psi$ is  coercive, and $0<\mu<\frac{2}{L_f}$. Let the sequence $(x^k)_{k\geq 0}$
be generated by  scheme (\ref{reweight}). If the sequence $(e^k)_{k\geq 0}$ satisfies
\begin{equation}
   \|e^k\|=\mathcal{O}(\frac{1}{k^{\alpha}}),\,\alpha>1.
\end{equation}
Then, the sequence $(x^k)_{k\geq 0}$ has finite length, i.e.
\begin{equation}
\sum_{k=0}^{+\infty}\|x^{k+1}-x^k\|<+\infty.
\end{equation}
And $\{x^k\}_{k=0,1,2,3,\ldots}$ converges to a critical point $x^*$ of $\Psi$.
\end{proposition}
\subsection{Inexact DC algorithm}
In this part, we consider nonconvex optimization problems of the following type
\begin{equation}\label{DC}
\min \{\Xi(x) = f(x) + g(x) - h(x)\},
\end{equation}
where $g$ is proper and lower semicontinuous, $f$ is differentiable with $L_f$-Lipschitz gradient, and $h$ is convex and differentiable  with $L_h$-Lipschitz gradient. Such a problem is discussed in \cite{mainge2008convergence}. If $f$ vanishes,
problem (\ref{DC}) will reduce to the DC programming  \cite{tao1997convex}
\begin{equation}
\min\{g(x)-h(x)\}.
\end{equation}
A novel DC algorithm is proposed in \cite{an2017convergence} for (\ref{DC}) and the convergence is also proved.
The inexact version of this algorithm can be expressed as
\begin{equation}\label{indc}
x^{k+1}= \textbf{prox}_{\gamma g}\left(x^k - \gamma(\nabla f(x^k) -\nabla h(x^k))+e^k\right),
\end{equation}
where $\gamma$ is the stepsize, and $e^k$ is the noise. The cautious reader may find that iteration (\ref{indc}) is actually a special case of (\ref{pga}) if regarding $f-h$ as a whole. But with the specific structure, iteration (\ref{indc}) enjoys more properties than (\ref{pga}), like larger stepsize. It is easy to see that $\nabla(f-h)=\nabla f-\nabla h$ is Lipchitz with $L_f+L_h$. If directly using the convergence results for (\ref{pga}) (Theorem \ref{converpga}), the stepsize $\gamma$ shall satisfy $\gamma<\frac{1}{L_f+L_h}$. However, a larger step can be selected for iteration (\ref{indc}); the stepsize can be $0<\gamma<\frac{2}{L_f}$ (Lemma \ref{dclemma1}).
\begin{lemma}\label{dclemma1}
Let $(x^{k})_{k\geq 0}$ be generated by scheme (\ref{indc}) and $0<\gamma<\frac{2}{L_f}$.
Then, we   have
\begin{equation}
    \Xi(x^k)-\Xi(x^{k+1})\geq (\frac{1}{\gamma}-\frac{L_{f}}{2})\|x^{k}-x^{k+1}\|^2-\frac{\|e^k\|^2}{\gamma(2-\gamma L_f)}.
\end{equation}
\end{lemma}
Direct computations yield
\begin{align}\label{dclema1-t1}
&\Xi(x^k)-\Xi(x^{k+1})=f(x^k)-f(x^{k+1})+g(x^k)-g(x^{k+1})+h(x^{k+1})-h(x^k)\nonumber\\
&\quad\quad\geq \langle \nabla f(x^{k}),x^k-x^{k+1}\rangle-\frac{L_f}{2}\|x^k-x^{k+1}\|_2^2+g(x^k)-g(x^{k+1})+\langle x^{k+1}-x^k,\nabla h(x^k)\rangle.
\end{align}
On the other hand,
with Lemma \ref{tool},  we have
\begin{equation}\label{dclema1-t2}
    \gamma g(x^{k+1})+\frac{\|x^k - \gamma(\nabla f(x^k) -\nabla h(x^k))+e^k-x^{k+1}\|^2}{2}\leq \gamma g(x^k)+\frac{\|- \gamma(\nabla f(x^k) -\nabla h(x^k))+e^k\|^2}{2}.
\end{equation}
Combining (\ref{dclema1-t1}) and (\ref{dclema1-t2}) gives
\begin{eqnarray}\label{rewlema1temp4}
\Xi(x^k)-\Xi(x^{k+1})&\geq&(\frac{1}{\gamma}-\frac{L_f}{2})\|x^k-x^{k+1}\|^2+\frac{\langle e^k,x^k-x^{k+1}\rangle}{\gamma}\nonumber\\
&\geq&\frac{1}{2}(\frac{1}{\gamma}-\frac{L_f}{2})\|x^k-x^{k+1}\|^2-\frac{\|e^k\|_2^2}{\gamma(2-\gamma L_f)},
\end{eqnarray}
where we use the  inequality $\langle e^k,x^k-x^{k+1}\rangle\geq -\frac{1}{2}(1-\frac{\gamma L_f}{2})\|x^k-x^{k+1}\|^2-\frac{\|e^k\|^2}{2-\gamma L_f}$.

\begin{lemma}\label{dclemma2}
Let $(x^{k})_{k\geq 0}$ be generated by scheme (\ref{indc}). Then, there exist $S,D>0$ such that
\begin{equation}
  \emph{dist}(\textbf{0}, \partial\Xi(x^{k+1}))\leq S\|x^{k+1}-x^{k}\|+D\|e^k\|.
\end{equation}
\end{lemma}
\begin{proof}
With scheme of the algorithm,
\begin{equation}\label{dclemma2-t1}
   \frac{x^k-x^{k+1}}{\gamma}- \nabla f(x^k) +\frac{e^k}{\gamma}+\nabla h(x^k)\in   \partial g(x^{k+1}).
\end{equation}
Thus, we have
\begin{equation}\label{dclemma2-t2}
   \frac{x^k-x^{k+1}}{\gamma}+ \nabla f(x^{k+1})- \nabla f(x^k) +\frac{e^k}{\gamma}+\nabla h(x^k)-\nabla h(x^{k+1})\in  \partial\Xi(x^{k+1}).
\end{equation}
Hence,
\begin{eqnarray}
    \textrm{dist}(\textbf{0},\partial\Xi(x^{k+1}))&\leq&\|\frac{x^k-x^{k+1}}{\gamma}+ \nabla f(x^{k+1})- \nabla f(x^k) +\frac{e^k}{\gamma}+\nabla h(x^k)-\nabla h(x^{k+1})\|\nonumber\\
    &\leq&(\frac{1}{\gamma}+L_f+L_h)\|x^k-x^{k+1}\|+\frac{1}{\gamma}\|e^k\|.
\end{eqnarray}
\end{proof}

\begin{lemma}
Let $(x^{k})_{k\geq 0}$ is generated by scheme (\ref{indc}) and $0<\gamma<\frac{2}{L_f}$, and $\Xi$ be coercive, and $e^k\rightarrow \textbf{0}$. Then, for $x^*$ being the stationary point of $(x^{k})_{k\geq 0}$, there exists a subsequence $(x^{k_j})_{j\geq 0}$ convergent to $x^*$ satisfying $\Xi(x^{k_j})\rightarrow \Xi(x^*)$ and $x^*\in \emph{crit}(\Xi)$.
\end{lemma}
\begin{proof}
With Lemma \ref{dclemma1}, $(x^k)_{k\geq 0}$ is bounded. For any $x^*\in\textrm{crit}(\Phi)$, there exists a subsequence $(x^{k_j})_{j\geq 0}$ converges to $x^*$. With Lemmas \ref{points} and \ref{dclemma1}, we also have
\begin{equation}
    x^{k_j-1}\rightarrow x^*.
\end{equation}
And in each iteration of updating $x^{k_j}$, with Lemma \ref{tool}, we have
\begin{eqnarray}
    \gamma g(x^{k_j})&+&\frac{\|x^{k_j-1} - \gamma(\nabla f(x^{k_j-1}) -\nabla h(x^{k_j-1})+e^{k_j-1}-x^{k_j}\|^2}{2}\nonumber\\
    &\leq& \gamma g(x^{*})+\frac{\|x^{k_j}- \gamma(\nabla f(x^{k_j-1}) -\nabla h(x^{k_j-1}))+e^{k_j-1}-x^*\|^2}{2}.
\end{eqnarray}
Taking $j\rightarrow+\infty$, we have
\begin{equation}
    \underset{j\rightarrow+\infty}{\lim\sup}~~g(x^{k_j})\leq g(x^*).
\end{equation}
And recalling the lower semi-continuity of $g$,
\begin{equation}
     g(x^*)\leq \underset{j\rightarrow+\infty}{\lim\inf}~~g(x^{k_j}).
\end{equation}
That means $\lim_j g(x^{k_j})=g(x^*)$; combining the continuity of $f$ and $h$, we then prove $\lim_{j}\Xi(x^{k_j})\rightarrow \Xi(x^*)$. In Lemma \ref{dclemma2}, substituting $k=k_j-1$, we can see $x^*$ is a critical point of $\Xi$.
\end{proof}

\begin{proposition}
Let $(x^{k})_{k\geq 0}$ be generated by scheme (\ref{indc}). Functions $f$, $g$ and $h$ are all semi-algebraic. And the stepsize satisfies $0<\gamma<\frac{2}{L_f}$, and $\Xi$ is coercive, and
$$\|e^k\|=\mathcal{O}(\frac{1}{k^{\alpha}}),\,\alpha>1.$$
Then, the sequence $(x^k)_{k\geq 0}$ has finite length, i.e.
\begin{equation}
\sum_{k=0}^{+\infty}\|x^{k+1}-x^k\|<+\infty.
\end{equation}
And $\{x^k\}_{k=0,1,2,3,\ldots}$ converges to a critical point $x^*$ of $\Xi$.
\end{proposition}
\subsection{Inexact nonconvex ADMM algorithm}
Alternating Direction Method of Multipliers (ADMM) \cite{gabay1976dual,glowinski1975approximation} is a powerful tool for the minimization of
composite functions with linear constraints.
An inexact nonconvex ADMM scheme is considered for the composite optimization
\begin{equation}
\min_{x,y}\{f(x)+g(y),~\textrm{s.t.}~~x+y=\textbf{0}\}\footnote{The reuslt can be easiy extended to a more general constraint $Ax+By=c$. Here, we consider this case just for the simplicity of presentation. },
\end{equation}
where $g$  is differentiable with $L_g$-Lipschitz gradient.
We consider the following inexact algorithm as
\begin{align}\label{inadmm}
\left\{\begin{array}{c}
         x^{k+1}=\textbf{prox}_{\frac{f}{\alpha+\beta}}(\frac{\alpha}{\alpha+\beta}x^k-\frac{\beta}{\alpha+\beta}y^k-\frac{\gamma^k}{\alpha+\beta}+e_1^{k+1}), \\
          y^{k+1}=\textbf{prox}_{\frac{g}{\beta}}(-x^{k+1}-\frac{\gamma^k}{\beta}+e_2^{k+1}),\\
         \gamma^{k+1}=\gamma^{k}+\beta(x^{k+1}+y^{k+1}),
       \end{array}
\right.
\end{align}
where the augmented Lagrangian function  $L_{\beta}$ is defined as
\begin{align}\label{L1}
L_{\beta}(x, y,\gamma)
=f(x)+g(y)+\langle\gamma, x+ y \rangle
+\frac{\beta}{2} \| x+y\|^{2},
\end{align}
where $\gamma$ is the Lagrangian dual variable and $\alpha>0$ is a parameter.
If $e^k_1\equiv \textbf{0}$ and $e^k_2\equiv \textbf{0}$, the scheme is the standard ADMM . Nonconvex ADMM has been frequently studied in recent years \cite{wang2015global,li2015global,li2016douglas,sun2017pre,sun2017admmdc,ames2016alternating,hong2016convergence}.
 First, we prove a critical lemma.
\begin{lemma}\label{clem}
Let $(x^k,y^k,\gamma^k)_{k\geq 0}$ be generated by (\ref{inadmm}), we then have
\begin{equation}
    \|\gamma^{k+1}-\gamma^k\|^2\leq \rho_1 \|y^{k+1}-y^k\|^2+\rho_2\|e^{k+1}_2-e^{k}_2\|^2,
\end{equation}
where $\rho_1=2L_g^2$, and $\rho_2=2\beta^2$.
\end{lemma}
\begin{proof}
The second step of each iteration gives
 \begin{equation}
 \nabla g(y^{k+1})=-[\gamma^k+\beta (x^{k+1}+y^{k+1})]+\beta e_2^{k+1}.
 \end{equation}
 With the fact $\gamma^{k+1}=\gamma^{k}+\beta(x^{k+1}+ y^{k+1})$, we then have
 \begin{equation}
    \nabla g(y^{k+1})=-\gamma^{k+1}+\beta e_2^{k+1}.
 \end{equation}
 Substituting  $k+1$ with $k$,
  \begin{equation}\label{bneed}
    \nabla g(y^{k})=-\gamma^{k}+\beta e_2^{k}.
 \end{equation}
 Substraction of the two equalities above yields
 \begin{eqnarray}
 \|\gamma^{k+1}-\gamma^k\|\leq L_g\|y^{k+1}-y^{k}\|+\beta\|e^{k+1}_2-e^{k}_2\|.
 \end{eqnarray}
\end{proof}
We define the auxiliary points as
$$d^k:= (x^k,y^k,\gamma^k,y^{k-1}), \omega^k:=(x^k,y^k), \varepsilon^k:=\left(\begin{array}{c}
                                                                 e^{k+1}_2-e^{k}_2 \\
                                                                 e^{k+1}_1 \\
                                                                 e^{k+1}_2
                                                               \end{array}
\right),$$
 and the Lyapunov function as
$$F(d)=F(x,y,\gamma) := L_{\beta}(x, y,\gamma).$$
  In the following, we prove the conditions (\ref{condition}) and \eqref{pointass} hold for $(d^k)_{k\geq 0}$ and $(\omega^k)_{k\geq 0}$  with respect to  $F$.
\begin{lemma}
If $ \sum_{k}\|\omega^k-\omega^{k+1}\|<+\infty$, we have $\sum_{k}\|d^k-d^{k+1}\|<+\infty.$
\end{lemma}
 \begin{proof}
 Direct basic algebraic computation gives the result.
 \end{proof}
\begin{lemma}\label{adlemma1}
Let $(d^{k})_{k\geq 0}$ is generated by scheme (\ref{inadmm}), and $g$ is convex,
\begin{equation}\label{admmcondition}
    \beta>(2\sqrt{2}+1)L_g,\,~ \alpha>0.
\end{equation}
Then, we will have
\begin{equation}
    F(d^k)-F(d^{k+1})\geq \nu\|\omega^{k}-\omega^{k+1}\|^2-\rho\|\varepsilon^{k}\|^2
\end{equation}
for some $\nu,\rho>0$.
\end{lemma}
\begin{proof}
Notice that $y^{k+1}$ is the minimizer of
$$\underbrace{f(x)+g(y^k)+\langle\gamma^k,x+y^k\rangle+\frac{\beta}{2}\|x+y^k-\frac{\alpha+\beta}{\beta}e^{k+1}_1\|^2}_{:=\hat{L}_{\beta}(x,y^k,\gamma^k)}+\frac{\alpha}{2}\|x^{k+1}-x^k\|^2.$$
Thus, we can have
\begin{equation*}
    \hat{L}_{\beta}(x^{k+1},y^{k},\gamma^k)+\frac{\alpha}{2}\|x^{k+1}-x^k\|^2\leq  \hat{L}_{\beta}(x^{k},y^{k},\gamma^k).
\end{equation*}
Direct computing the yields
\begin{equation}
    L_{\beta}(x^{k+1},y^{k},\gamma^k)+\frac{\alpha}{2}\|x^{k+1}-x^k\|^2\leq L_{\beta}(x^{k+1},y^k,\gamma^k)+(\alpha+\beta)\langle e^{k+1}_1,x^k-x^{k+1}\rangle.
\end{equation}
With the inequality
\begin{equation}
    (\alpha+\beta)\langle e^{k+1}_1,x^k-x^{k+1}\rangle\leq\frac{(\alpha+\beta)^2}{\alpha}\|e_1^{k+1}\|^2+\frac{\alpha}{4}\|x^{k+1}-x^k\|^2,
\end{equation}
we then have
\begin{eqnarray}
    L_{\beta}(x^{k+1},y^{k},\gamma^k)+\frac{\alpha}{4}\|x^{k+1}-x^k\|^2
    \leq L_{\beta}(x^{k},y^k,\gamma^k)+\frac{(\alpha+\beta)^2}{\alpha}\|e_1^{k+1}\|^2.
\end{eqnarray}

Similarly, note that $y^{k+1}$ is the minimizer of
$$\tilde{L}_{\beta}(x^{k+1},y,\gamma^k):=f(x^{k+1})+g(y)+\langle\gamma^k,x^{k+1}+y\rangle+\frac{\beta}{2}\|x^{k+1}+y-e^{k+1}_2\|^2,$$
and $\tilde{L}_{\beta}(x^{k+1},y,\gamma^k)$ is strongly convex with constant $\beta-L_g$.  Thus, we have
\begin{equation*}
    \tilde{L}_{\beta}(x^{k+1},y^{k+1},\gamma^k)+\frac{\beta-L_g}{2}\|y^{k+1}-y^k\|^2\leq  \tilde{L}_{\beta}(x^{k+1},y^{k},\gamma^k).
\end{equation*}
After simplifications, we then derive
\begin{equation}\label{adlemma1-t-1}
    L_{\beta}(x^{k+1},y^{k+1},\gamma^k)+\frac{\beta-L_g}{2}\|y^{k+1}-y^k\|^2\leq L_{\beta}(x^{k+1},y^k,\gamma^k)+\beta\langle e^{k+1}_2,y^k-y^{k+1}\rangle.
\end{equation}
By using the inequality
\begin{equation}
    \beta\langle e^{k+1}_2,y^k-y^{k+1}\rangle\leq\frac{\beta^2}{\beta-L_g}\|e_2^{k+1}\|^2+\frac{\beta-L_g}{4}\|y^{k+1}-y^k\|^2.
\end{equation}
With (\ref{adlemma1-t-1}),
we have
\begin{equation}
    L_{\beta}(x^{k+1},y^{k+1},\gamma^k)+\frac{\beta-L_g}{4}\|y^{k+1}-y^k\|^2\leq L_{\beta}(x^{k+1},y^k,\gamma^k)+\frac{\beta^2}{\beta-L_g}\|e_2^{k+1}\|^2.
\end{equation}

With Lemma \ref{clem},
\begin{eqnarray}
    L_{\beta}(x^{k+1},y^{k+1},\gamma^{k+1})&=&L_{\beta}(x^{k+1},y^{k+1},\gamma^{k})+\frac{\|\gamma^{k+1}-\gamma^k\|^2}{\beta}\nonumber\\
    &\leq&L_{\beta}(x^{k+1},y^{k+1},\gamma^{k})+\frac{\rho_1}{\beta} \|y^{k+1}-y^k\|^2+\frac{\rho_2}{\beta}\|e^{k+1}_2-e^{k}_2\|^2.
\end{eqnarray}
Thus, we have
\begin{eqnarray}
    F(d^{k+1})&+&\frac{\alpha}{4}\|x^{k+1}-x^k\|^2
    +[\frac{\beta-L_g}{4}-\frac{\rho_1}{\beta}]\|y^{k+1}-y^k\|^2\nonumber\\
    &-&\max\{\frac{(\alpha+\beta)^2}{\alpha},\frac{\beta^2}{\beta-L_g},\frac{\rho_2}{\beta}\}\cdot\|\varepsilon^k\|^2\leq F(d^k).
\end{eqnarray}
Letting $\nu:=\min\{\frac{\alpha}{4},\frac{\beta-L_g}{4}-\frac{\rho_1}{\beta}\}$ and $\rho:=\max\{\frac{(\alpha+\beta)^2}{\alpha},\frac{\beta^2}{\beta-L_g},\frac{\rho_2}{\beta}\}$, we then prove the result.
\end{proof}
\begin{lemma}\label{adlemma2}
There exist $S,D>0$ such that
\begin{equation}
  \emph{dist}(\textbf{0}, \partial F(d^{k+1}))\leq S\|\omega^{k+1}-\omega^{k}\|+D\|\varepsilon^k\|.
\end{equation}
\end{lemma}
\begin{proof}
From Lemma \ref{clem}, we have
\begin{equation}\label{adlemma2-t0}
    \|\gamma^{k+1}-\gamma^k\|\leq \sqrt{\rho_1} \|y^k-y^{k+1}\|+\sqrt{\rho_3}\|e^{k+1}_2-e^{k}_2\|.
\end{equation}
The optimization condition for updating  $x^{k+1}$ is
\begin{equation}\label{adlemma2-t1}
     \alpha(x^k-x^{k+1})-\beta(x^{k+1}+y^k)-\gamma^k+(\alpha+\beta)e^k_1\in\partial f(x^{k+1}).
\end{equation}
With direct calculation, we have
\begin{equation}\label{adlemma2-t2}
    \partial_{x} F(d^{k+1})=\partial f(x^{k+1})+\gamma^{k+1}
+\beta (x^{k+1}+y^{k+1})
\end{equation}
Thus, we have
\begin{eqnarray}\label{adlemma2-t3}
    \textrm{dist}[\textbf{0},\partial_{x} F(d^{k+1})]&\leq& \|\alpha(x^k-x^{k+1})
+\beta (y^{k+1}-y^k)+(\alpha+\beta)e^k_1\|\nonumber\\
&\leq&S_x (\|\omega^{k+1}-\omega^k\|)+D_x(\|e^k_1\|),
\end{eqnarray}
where $S_x=\max\{\alpha,\,~\beta\}$ and $D_x=\alpha+\beta$. While in updating $y^{k+1}$, we have
\begin{equation}\label{adlemma2-t4}
-[\gamma^k+\beta (x^{k+1}+y^k)]+\beta e_2^{k+1}= \nabla g(y^{k+1}).
\end{equation}
And we have
\begin{equation}\label{adlemma2-t5}
    \partial_{y} F(d^{k+1})=\nabla g(y^{k+1})+\gamma^{k+1}
+\beta (x^{k+1}+y^{k+1})
\end{equation}
Combining (\ref{adlemma2-t4}) and (\ref{adlemma2-t5}),
\begin{eqnarray}\label{adlemma2-t6}
        \textrm{dist}[\textbf{0},\partial_{y} F(d^{k+1})]&\leq& \|
(\gamma^{k+1}-\gamma^k)+\beta (y^{k+1}-y^k)+\beta e_2^{k+1}\|\nonumber\\
&\leq&\beta\|y^{k+1}-y^k\|+\sqrt{\rho_1}\|y^k-y^{k+1}\|+\sqrt{\rho_2}\|e^{k+1}_2-e^{k}_2\|+\beta\|e_2^{k+1}\|\nonumber\\
&\leq&S_y (\|\omega^{k+1}-\omega^k\|)+D_y (\|e^{k+1}_2-e^{k}_2\|+\|e_2^{k+1}\|),
\end{eqnarray}
where $S_y=\beta+\sqrt{\rho_1}$ and $D_y=\max\{\sqrt{\rho_2},\beta\}$.
Noting
\begin{equation}
    \partial_{\gamma} F(d^{k+1})=x^{k+1}+y^{k+1}=\frac{\gamma^{k+1}-\gamma^k}{\beta},
\end{equation}
we have
\begin{eqnarray}
    \textrm{dist}(\textbf{0},\partial_{\gamma} F(d^{k+1}))&\leq&\frac{\|\gamma^{k+1}-\gamma^k\|}{\beta}\leq\frac{\sqrt{\rho_1}}{\beta} \|y^k-y^{k+1}\|+\frac{\sqrt{\rho_2}}{\beta}\|e^{k+1}_2-e^{k}_2\|\nonumber\\
    &\leq&S_{\gamma}(\|\omega^{k+1}-\omega^k\|)+D_{\gamma}\|e^{k+1}_2-e^{k}_2\|,
\end{eqnarray}
where $S_{\gamma}=\frac{\sqrt{\rho_1}}{\beta}$ and $D_{\gamma}=\frac{\sqrt{\rho_2}}{\beta}$.
Letting $S=S_{x}+S_{y}+S_{\gamma}$ and $D=D_{x}+D_{y}+D_{\gamma}$, we then prove the result.
\end{proof}

Then, we prove  $\inf F(d^k)>-\infty$. Then, we can obtain the boundedness of the points.
\begin{lemma}\label{adboundness}
If  there exists $\sigma_0>0$ such that
\begin{align}\label{quad}
\inf\{g(y)-\sigma_0\|\nabla g(y)\|^2\}>-\infty.
\end{align}
We also assume that $(e^k_2)_{k\geq 0}$ is bounded and condition (\ref{admmcondition}) holds, and $f(x)$ is coercive. If
$$\beta\geq\frac{1}{\sigma_0},$$
then, the sequence $\{d^k\}_{k=0,1,2,\ldots}$ is bounded.
\end{lemma}
\begin{proof}
From (\ref{bneed}),
 \begin{equation}\label{bneed2}
    \|\gamma^{k}\|^2\leq2\|\nabla g(y^{k})\|^2+2\beta^2\|e_2^{k}\|^2.
 \end{equation}
We have
\begin{align}
F(d^k)&=f(x^k)+g(y^k)+\langle \gamma^k,x^k+y^k\rangle+\frac{\beta}{2}\|x^k+y^k\|^2\nonumber\\
&=f(x^k)+g(y^k)-\frac{\|\gamma^k\|^2}{2\beta}+\frac{\beta}{2}\|x^k+y^k+\frac{\gamma^k}{\beta}\|^2\nonumber\\
&=f(x^k)+g(y^k)-\frac{\sigma_0}{2}\|\gamma^k\|^2+(\frac{\sigma_0}{2}-\frac{1}{2\beta})\|\gamma^k\|^2+\frac{\beta}{2}\|x^k+y^k+\frac{\gamma^k}{\beta}\|^2\nonumber\\
(\ref{bneed2})&\geq f(x^k)+g(y^k)-\sigma_0\|\nabla g(y^k)\|^2\nonumber\\
&+(\frac{\sigma_0}{2}-\frac{1}{2\beta})\|\gamma^k\|^2+\frac{\beta}{2}\|x^k+y^k+\frac{\gamma^k}{\beta}\|^2-\sigma_0\beta^2\|e_2^k\|^2.
\end{align}
Noting $\lim_k\|e_2^k\|^2=0$, we then can see $\{f(x^k)\}_{k=0,1,2,\ldots}$, $\{\gamma^k\}_{k=0,1,2,\ldots}$, $\{x^k+y^k+\frac{\gamma^k}{\beta}\}_{k=0,1,2,\ldots}$ are all bounded. Then, $\{d^k\}_{k=0,1,2,\ldots}$ is bounded.
\end{proof}
\begin{remark}
Combining (\ref{admmcondition}), we need to set
\begin{equation}
    \beta> \min\{(2\sqrt{2}+1)L_g,\,~\frac{1}{\sigma_0}\},\,~  \alpha>0.
\end{equation}
\end{remark}

\begin{remark}
The condition (\ref{quad}) holds for many quadratical functions \cite{li2015global,sun2017pre}. This condition also implies the function $g$ is similar to quadratical function and its property is ``good".
\end{remark}

\begin{lemma}
Let $(d^{k})_{k\geq 0}$ is generated by scheme (\ref{inadmm}), and $e^k_1\rightarrow \textbf{0}$, and $e^k_2\rightarrow \textbf{0}$. And let  conditions of Lemmas \ref{adlemma1} and \ref{adboundness} hold. Then, for any stationary point $d^*$ of $(d^k)_{\geq 0}$, there exists a subsequence $(d^{k_j})_{j\geq 0}$ converges to $d^*$ satisfying $F(d^{k_j})\rightarrow F(d^*)$ and $d^*\in \emph{crit}(F)$.
\end{lemma}
\begin{proof}
Obviously, we have $e^k_1\rightarrow \textbf{0}$.
With Lemma \ref{adboundness}, $(d^k)_{k\geq 0}$ is bounded; so are $(x^k)_{k\geq 0}$  and $(y^k)_{k\geq 0}$.  For any stationary point $d^*=(x^*,y^*,\gamma^*,y^*)$, there exists a subsequence $(d^{k_j})_{j\geq 0}$ converges to $d^*$. With Lemmas \ref{points} and \ref{adlemma1}, we also have
\begin{equation}
    x^{k_j-1}\rightarrow x^*, y^{k_j-1}\rightarrow y^*.
\end{equation}
Noting $\textbf{0}\in \partial F(d^*)$, $x^*+y^*=\textbf{0}$; thus, $\gamma^{k_j-1}\rightarrow \gamma^*$.
And in each iteration of updating $x^{k_j}$, with Lemma \ref{tool}, we have
\begin{eqnarray}
    \frac{1}{\alpha+\beta} f(x^{k_j})&+&\frac{\|\frac{\alpha}{\alpha+\beta}x^{k_j-1}-\frac{\beta}{\alpha+\beta}y^{k_j-1}-\frac{\gamma^{k_j-1}}{\alpha+\beta}+e_1^{k_j}-x^{k_j}\|^2}{2}\nonumber\\
    &\leq& \frac{1}{\alpha+\beta} f(x^{*})+\frac{\|\frac{\alpha}{\alpha+\beta}x^{k_j-1}-\frac{\beta}{\alpha+\beta}y^{k_j-1}-\frac{\gamma^{k_j-1}}{\alpha+\beta}+e_1^{k_j}-x^*\|^2}{2}.
\end{eqnarray}
Taking $j\rightarrow+\infty$, we have
\begin{equation}
    \underset{j\rightarrow+\infty}{\lim\sup}~~f(x^{k_j})\leq f(x^*).
\end{equation}
And recalling the lower semi-continuity of $f$,
\begin{equation}
     f(x^*)\leq \underset{j\rightarrow+\infty}{\lim\inf}~~f(x^{k_j}).
\end{equation}
That means $\lim_j f(x^{k_j})=f(x^*)$; combining the continuity of $g$, we then prove the continuity condition. On the other hand, Lemma \ref{adlemma2} can guarantee that  $d^*\in \textrm{crit}(F)$.
\end{proof}
Finally, we present the convergence result for the inexact ADMM (\ref{inadmm}).
\begin{proposition}
Let $(x^{k})_{k\geq 0}$ be generated by scheme (\ref{inadmm}) and   conditions of Lemmas \ref{adlemma1} and \ref{adboundness} hold. Assume that $f$ and $g$ are both semi-algebraic.
If
$$\|e^k_1\|+\|e^k_2\|=\mathcal{O}(\frac{1}{k^{\alpha}}),\,\alpha>1,$$
then, the sequence $(x^k,y^k)_{k\geq 0}$ has finite length, i.e.
\begin{equation}
\sum_{k=0}^{+\infty}(\|x^{k+1}-x^k\|+\|y^{k+1}-y^k\|)<+\infty.
\end{equation}
\end{proposition}
\begin{proof}
Noting
\begin{eqnarray}
   \|\varepsilon^k\|\leq 2\|e^k_1\|+2\|e^k_2\|,
\end{eqnarray}
that is also
\begin{eqnarray}
   \|\varepsilon^k\|=\mathcal{O}(\frac{1}{k^{\alpha}}),\,\alpha>1.
\end{eqnarray}
With the lemmas proved in this part and Theorem \ref{conver}, we then prove the result.
\end{proof}
\section{Conclusion}
In this paper, we build a framework to prove the sequence convergence for  inexact nonconvex and nonsmooth algorithms. The sequence generated by the algorithm is proved to  converge to a critical point of the objective function under finite energy assumption on the noise and the K{\L} property assumption. We apply our theoretical results to several  specific  algorithms and obtain the specific convergence results.
\section*{Acknowledgments}
The authors are thankful to the editors and anonymous referees for their useful suggestions. We are grateful for the support from the National Key Research and Development Program of China (2017YFB0202003), and  National Science Foundation of China 
(No. 61603322)  and (No.61571008), and National Natural
Science Foundation of Hunan (2018JJ3616).

\end{document}